\documentclass{article}
\usepackage{amsthm}
\usepackage{fancyhdr}
\usepackage{amsmath}
\usepackage{indentfirst}
\usepackage{titlesec}
\usepackage{amsfonts}
\usepackage{amssymb}

\newcommand{\cA}{\mathcal A}
\newcommand{\cB}{\mathcal B}

\newcommand{\cF}{\mathcal F}

\newcommand{\cL}{\mathcal L}
\newcommand{\cO}{\mathcal O}
\newcommand{\cR}{\mathcal R}
\newcommand{\cS}{\mathcal S}

\newcommand{\N}{\mathbb N}
\newcommand{\R}{\mathbb R}
\newcommand{\Z}{\mathbb Z}
\newcommand{\doub}{\textrm{\rm doub}}
\newcommand{\cov}{\textrm{\rm Cov}}
\newcommand{\corr}{\textrm{\rm Corr}}

 \newtheorem{theorem}{Theorem}[section]

 \newtheorem{lemma}[theorem]{Lemma}

 \newtheorem{corollary}[theorem]{Corollary}

 \newtheorem{proposition}[theorem]{Proposition}
 \theoremstyle{definition}
 \newtheorem{definition}[theorem]{Definition}

\newtheorem{remark}[theorem]{Remark}

 \newtheorem{notations}[theorem]{Notations}

 \newtheorem{conventions}[theorem]{Conventions}

 \newtheorem{example}[theorem]{Example}

 \newtheorem{context}[theorem]{Context}

 \numberwithin{equation}{section}

 \numberwithin{theorem}{section}

\begin{document}

\title{\textbf{An exposition of some basic features of strictly stationary, reversible Markov chains}}

\newcommand{\orcidauthorA}{0000-0000-000-000X} 

\author{Richard C.\ Bradley \\
Department of Mathematics, 
Indiana University, Bloomington, IN 47405, U.S.A.\\ bradleyr@indiana.edu}


\maketitle

\begin{abstract}It has been well known for some time that for strictly stationary Markov chains that are ``reversible'', the special symmetry (with the distribution of the Markov chain as a whole being invariant under a reversal of the ``direction of time'') provides special extra features in the mathematical theory. This paper here is primarily a purely expository review of some of the basic aspects of that special theory. The mathematical techniques employed in this review are relatively gentle, involving only some basic measure-theoretic probability theory. With the uncertain possible exception of the material in the final section, everything in this paper has long been well known, either explicitly or in a (perhaps somewhat hidden) implicit form.
\end{abstract}










\section{Introduction}
\label{sc1}
Over the past few decades, there has been considerable interest in the ``structures'', limit theory, and applications of strictly stationary Markov chains $X := (X_k$, $k\in \Z)$ that are ``reversible''--i.e. the distribution of the ``time-reversed'' Markov chain $(X_{-k}$, $k\in \Z)$ is the same as that of the Markov chain $X$ itself.  See for example the papers of Roberts and Rosenthal \cite{ref-journal-RR}, Roberts and Tweedie  \cite{ref-journal-RT} or the more recent papers of Longla and Peligrad  \cite{ref-journal-LP} and Longla  \cite{ref-journal-Longla17}, and the references in all of these papers.

This paper here is intended to give, in a ``leisurely'' form---that is, with generous detail and an intended ``generous absence of sophistication''---an exposition of a few standard basic features of strictly stationary Markov chains that are ``reversible''.  This paper is intended to be purely expository. We believe that practically everything in this paper---practically every concept, definition, statement, and proof---has already been in the literature for some time, in various papers by various researchers, either in transparent explicit forms or in disguised implicit forms.

A ``core purpose'' of this paper is to provide a ``leisurely'' exposition of a certain well known result that will be informally stated here under the heading ``Theorem \ref{th1.1}.1'' in quotation marks.

\medskip

\noindent \textbf{``Theorem 1.1.''} \label{th1.1}
\textit{Suppose $X := (X_k$, $k\in \Z)$ is a strictly stationary Markov chain with real state space, such that $X$ is ``reversible'' and also ``irreducible'' (in a sense related to---in fact equivalent to---``Harris recurrence''). Then that Markov chain $X$ satisfies ``geometric ergodicity'' if and only if it satisfies a certain ``$\cL^2$ spectral gap'' condition.}
\medskip

That result, in a precise form that was more general and more detailed in several different ways (including for example a more general state space), was given by Roberts and Tweedie [\cite{ref-journal-RT}, Theorem 2]   (see also page 39, lines 11-15 in that paper); it built directly on earlier results of Roberts and Rosenthal \cite{ref-journal-RR}.  A precise formulation of ``Theorem \ref{th1.1}.1''---again, in a restricted form, rather than in full generality---will be included in Remark \ref{rm4.7} in Section \ref{sc4} later on.

In fact the purpose of this paper is to provide, in a ``leisurely'' manner, (i)~a review of relevant terminology and background mathematics, (ii)~a review of a precise formulation of ``Theorem \ref{th1.1}.1'' (in a restricted form) and certain other related statements involving strictly stationary, reversible Markov chains, and (iii)~a review of the arguments by which such statements are proved.

It is well known that the ``geometric ergodicity'' and ``$\cL^2$ spectral gap'' conditions in ``Theorem \ref{th1.1}.1'' can each be formulated in terms of the dependence coefficients associated with certain ``strong mixing conditions''. See for example the papers of Longla and Peligrad \cite{ref-journal-LP}, Longla \cite{ref-journal-Longla13}, and Longla \cite{ref-journal-Longla14}. Accordingly, the terminology and techniques associated with those mixing conditions and their dependence coefficients will provide the primary ``vehicle'' for the exposition in this paper. Those mixing conditions and their dependence coefficients will be reviewed in Section \ref{sc2}.

The ``main mathematical framework'' for this exposition will be set up in Section \ref{sc3}. It will actually involve strictly stationary Markov chains that are reversible but \textit{not}\ necessarily ``irreducible'' (in the sense related to ``Harris recurrence''). In the latter part of Section \ref{sc3}, the review will be ``crystallized'' in the form of three technical statements that are given together in Proposition \ref{pr3.5}.  A review of the proofs of those three statements will be given respectively in Sections \ref{sc5}, \ref{sc7}, and \ref{sc6} (in that order).

Section \ref{sc4} will start (through Remark \ref{rm4.5}) with a brief review of basic features of strictly stationary Markov chains (reversible or not) that are ``irreducible'' (in the sense related to ``Harris recurrence''), and then review (in Remark \ref{rm4.7}) a precise formulation (in a restricted form) of ``Theorem \ref{th1.1}.1'', as part of a broader exposition of related material in Corollary \ref{co4.6} and Remarks \ref{rm4.8}--\ref{rm4.11}.

\section{Preliminaries, including some mixing conditions}\label{sc2}
This section will start by laying out some basic notations and conventions that will be used in the rest of this paper. Then the rest of this section will be devoted to a review of three relevant mixing conditions.

\begin{notations}\label{no2.1} In this paper, the following notations will be used:

(A) The following standard symbols will be used:

\noindent $\R$ denotes the set of all real numbers;

\noindent $\cR$ denotes the Borel $\sigma$-field on $\R$;

\noindent $\Z$ denotes the set of all integers; and

\noindent $\N$ denotes the set of all positive integers.

\noindent The usual notations such as $\R^\Z$ and $\cR^\Z$ will be used for Cartesian products and product $\sigma$-fields.

(B) If $(a_1,a_3,a_3,\dots)$ is a sequence of nonnegative real numbers and $(b_1,b_2,b_3,\dots)$ is a sequence of positive numbers, then the notation ``$a_n \ll b_n$ as $n\to\infty$'' means that $a_n = \cO(b_n)$ as $n\to\infty$, that is, $\limsup_{n\to\infty} (a_n/b_n)<\infty$.

(C) If $(a_1,a_2,a_3,\dots)$ is a sequence of nonnegative real numbers, then the notation
\[
a_n \longrightarrow 0\,\,\mbox{at least exponentially fast as $n\to\infty$}
\]
means that there exists a number $r$ satisfying $0<r<1$ such that $a_n \ll r^n$ as $n\to\infty$.

(D) For a given nonnegative integer $n$, define the positive integer
\begin{equation}\label{eq2.1.1}
\doub(n) := 2^n.
\end{equation}
The letters ``$\doub$'' stand for ``doubling''. When an integer of the form $2^n$ $(n\in\{0,1,2,\dots\})$ occurs in a subscript or exponent, it will be written as $\doub(n)$ for typographical convenience.

(E) We shall follow the possibly old-fashioned convention that for two given sets $A$ and $B$, the notation $A\subset B$ simply means that any element of $A$ is also an element of $B$. In particular, $A=B$ (that is, the sets are identical) if and only if $A\subset B$ and $B\subset A$.
\end{notations}

\begin{conventions}\label{cn2.2} Throughout this paper, the following conventions will be used:

(A) The setting for the work in this paper is a probability space $(\Omega, \cF,P)$, rich enough to accommodate all random variables declared.  All random variables in this paper are defined on that probability space.

(B) Except where specified otherwise, the random variables in this paper are assumed to be ``$S$-valued'', where $(S,\cS)$ is an unspecified measurable space. That is (see (A) above), an ``$S$-valued random variable'' is a function $\zeta:\Omega \to S$ such that for every set $A\in \cS$, the set $\{\omega \in \Omega: \zeta (\omega) \in A\}$ is a member of the $\sigma$-field $\cF$.

(B*) In Section \ref{sc4}, the random variables will be real-valued, (that is, with $(S,\cS) = (\R,\cR))$. In Example \ref{ex2.8} in Section \ref{sc2}, the random variables in a certain Markov chain will take their values in the set $\{1,2,3,4\}$. Elsewhere in this paper, there will be many \textit{real}-valued random variables of the form $g(Y)$ where $Y$ is an $S$-valued random variable and $g:S\to\R$ is an $\cS/\cR$-measurable function---that is, for every $B\in \cR$, the set $\{s\in S: g(s) \in B\}$ is a member of the $\sigma$-field $\cS$.

(C) All (``two-sided'') sequences $(X_k$, $k\in \Z)$ of random variables (typically $S$-valued) considered in this paper, are strictly stationary.

(D) In particular, all Markov chains $(X_k$, $k\in \Z)$ considered in this paper are strictly stationary. Except when specified otherwise, the Markov chains will have ``state space $S$''---that is, the random variables $X_k$ will be $S$-valued (see (B) above).

(E) Refer to (A) and (B) again. For any $S$-valued random variable $\zeta$ on $(\Omega, \cF,P)$, let $\sigma(\zeta)$ denote the $\sigma$-field on $\Omega$ that is generated by~$\zeta$. By a well known generalization of Billingsley [\cite{ref-journal-Bill}, Theorem 20.1(i)], $\sigma(\zeta)$ consists precisely of the events $\{\zeta\in A\}$ where $A\in S$.

(F) Refer to (E) above. For any given family $(X_i, i\in I)$ of $S$-valued random variables, where $I$ is a nonempty index set, the notation $\sigma(X_i, i\in I)$ means the $\sigma$-field on $\Omega$ that is generated by this family. This is the smallest $\sigma$-field in which for each $i\in I$, $\sigma(X_i)$ is a sub-$\sigma$-field. That is, it is the smallest $\sigma$-field that contains as members all of the events $\{X_i\in A\}$ where $i\in I$ and $A\in \cS$.

(G) Refer to the (unspecified) measurable space $(S,\cS)$ in (B) again. For any set $A\in \cS$, the indicator function of $A$ (on $S$) will be denoted $I_A$---that is, the function $I_A: S \to\{0,1\}$ such that $I_A(s) = 1$ for $s\in A$ and $I_A(s) =0$ for $s\in S-A$.
\end{conventions}

\begin{definition}[Three measures of dependence.]\label{df2.3} Refer to Convention \ref{cn2.2}(A).

Suppose $\cA$ and $\cB$ are any two $\sigma$-fields $\subset \cF$.

Define the measure of dependence
\begin{equation}\label{eq2.3.1}
\alpha(\cA,\cB) := \sup_{A\in \cA, B\in \cB} |P(A\cap B)-P(A) P(B)|.
\end{equation}
Next, define the ``maximal correlation coefficient''
\begin{equation}
\label{eq2.3.2}
\rho(\cA,\cB) := \sup | \corr(Y,Z)|
\end{equation}
where the supremum is taken over all pairs of real-valued square-integrable random variables $Y$ and $Z$ such that $Y$ is $\cA$-measurable and $Z$ is $\cB$-measurable. Finally, define the measure of dependence
\begin{equation}\label{eq2.3.3}
\beta(\cA,\cB):= \sup \frac{1}{2} \sum^I_{i=1}\sum^J_{j=1} |P(A_i\cap B_j) - P(A_i)P(B_j)|
\end{equation}
where the supremum is taken over all pairs of finite partitions $\{A_1,A_2,\dots,A_I\}$ and $\{B_1,B_2,\dots,B_J\}$ of $\Omega$ such that $A_i\in \cA$ for each $i$ and $B_j\in \cB$ for each~$j$.
\end{definition}

\begin{remark}\label{rm2.4}
Suppose $\cA$ and $\cB$ are any two $\sigma$-fields $\subset \cF$.

(A) The following inequalities are elementary and well known:
\begin{equation}\label{eq2.4.1}
0\le 4\alpha(\cA,\cB) \le \rho(\cA,\cB) \le 1; \,\,\mbox{and}
\end{equation}
\begin{equation}\label{eq2.4.2}
0\le 2\alpha(\cA,\cB) \le \beta(\cA,\cB) \le 1.
\end{equation}
(See e.g.\ [\cite{ref-journal-Bradley2007}, Vol. 1, Proposition 3.11].)  The quantities $\alpha(\cA,\cB)$, $\rho(\cA,\cB)$, and $\beta(\cA,\cB)$ are all equal to $0$ if the $\sigma$-fields $\cA$ and $\cB$ are independent, and are all positive otherwise.

(B) By well known elementary arguments, the maximal correlation coefficient $\rho(\cA,\cB)$ has the following properties: First,
\begin{equation}\label{eq2.4.3}
\rho(\cA,\cB) = \sup \frac{\| E(Y|\cB)\|_2}{\|Y\|_2}
\end{equation}
where the supremum is taken over all real-valued, square-integrable, $\cA$-measurable random variables $Y$ such that $EY=0$. (In (\ref{eq2.4.3}) when necessary, interpret $0/0 := 0$.) Second,
\begin{equation}\label{eq2.4.4}
\rho(\cA,\cB) = \sup E(UV)
\end{equation}
where the supremum is taken over all pairs of real-valued, simple random variables $U$ and $V$ such that $U$ is $\cA$-measurable, $V$ is $\cB$-measurable, $EU = EV=0$, and $E(U^2) \le 1$ and $E(V^2) \le 1$.  In (\ref{eq2.4.4}), one might typically stipulate the equalities $E(U^2) = E(V^2) =1$ instead of the inequalities $E(U^2) \le 1$ and $E(V^2) \le 1$; however, that trivially does not affect the supremum there. Also in (\ref{eq2.4.4}), one might typically use $|E(UV)|$ instead of $E(UV)$; but again that trivially does not affect the supremum there, since one can always (say) replace $U$ by~$-U$.
\end{remark}

\begin{definition}[Three mixing condition] \label{df2.5}  Suppose $X:= (X_k$, $k\in \Z)$ is a (not necessarily Markovian) strictly stationary sequence of $S$-valued random variables.

(A) Refer to Convention \ref{cn2.2}(E). For each integer $j$, define the notations $\cF^j_{-\infty} :=\sigma(X_k, k\le j)$ and $\cF^\infty_j:= \sigma(X_k$, $k\ge j)$.

(B) For each positive integer $n$, define the following three ``dependence coefficients'':
\begin{align}
\label{eq2.5.1}
\alpha(n) &= \alpha(X,n) := \alpha(\cF^0_{-\infty}, \cF^\infty_n); \\
\label{eq2.5.2}
\rho(n) &= \rho(X,n) := \rho(\cF^0_{-\infty}, \cF^\infty_n);\,\,\mbox{and}\\
\label{eq2.5.3}
\beta(n) &= \beta(X,n) := \beta(\cF^0_{-\infty}, \cF^\infty_n).
\end{align}

(C) The strictly stationary sequence $X$ is said to satisfy

\noindent ``strong mixing'' (or ``$\alpha$-mixing'') if $\alpha(X,n) \to 0$ as $n\to\infty$;

\noindent ``$\rho$-mixing'' if $\rho(X,n) \to 0$ as $n\to \infty$;

\noindent ``absolute regularity'' (or ``$\beta$-mixing'') if $\beta(X,n) \to 0$ as $n\to\infty$.

The strong mixing ($\alpha$-mixing) condition is due to Rosenblatt \cite{ref-journal-Rosenblatt1956}. The $\rho$-mixing condition is due to Kolmogorov and Rozanov \cite{ref-journal-KolRoz}. (The ``maximal correlation coefficient'' $\rho(\cA,\cB)$ itself, for $\sigma$-fields $\cA$ and $\cB$, was first studied earlier by Hirschfeld \cite{ref-journal-Hirschfeld}, in a statistical context that had no particular connection with ``stochastic processes''.) The absolute regularity condition was first studied by Volkonskii and Rozanov \cite{ref-journal-VolkRoz}, and was attributed there to Kolmogorov.
\end{definition}

\begin{remark} \label{rm2.6}   Suppose $X:= (X_k, k\in \Z)$ is a (not necessarily Markovian) strictly stationary sequence of $S$-valued random variables

(A) Refer to both (A) and (B) (including eqs. (\ref{eq2.5.1})---(\ref{eq2.5.3})) in Definition \ref{df2.5}.

For each positive integer $n$, one  has by strict stationarity that $\alpha(n) = \alpha(\cF^j_{-\infty},\cF^\infty_{j+n})$ for every integer $j$, and the analogous comment holds for $\rho(n)$ and $\beta(n)$ as well.
\medskip

(B) One (trivially) has that $\alpha(1) \ge \alpha(2)\ge \alpha(3)\ge \dots\,$; and the analogous comment holds for the $\rho(n)$'s and for the $\beta(n)$'s.
\medskip

(C) By (\ref{eq2.4.1}) and (\ref{eq2.4.2}), one has that for each positive integer $n$,
\begin{align}
\label{eq2.6.1}
0&\le 4\alpha(n) \le \rho(n) \le 1,\,\,\mbox{and}\\
\label{eq2.6.2}
0 &\le 2\alpha(n) \le \beta(n) \le 1.
\end{align}
By (\ref{eq2.6.1}), $\rho$-mixing implies strong mixing ($\alpha$-mixing); and by (\ref{eq2.6.2}), absolute regularity ($\beta$-mixing) implies strong mixing.
\end{remark}

\begin{remark}[Markov chains] \label{rm2.7}   Suppose now that $X:=(X_k, k\in \Z)$ is a strictly stationary {\em Markov chain}, with state space $S$.

(A) As a consequence of the Markov property, for each positive integer $n$, eqs. (\ref{eq2.5.1})--(\ref{eq2.5.3}) hold in the following augmented forms for the given (strictly stationary) Markov chain $X$:
\begin{align}
\label{eq2.7.1}
\alpha(n) &= \alpha(\cF^0_{-\infty}, \cF^\infty_n) = \alpha(\sigma(X_0),\sigma(X_n));\\
\label{eq2.7.2}
\rho(n) &=\rho(\cF^0_{-\infty}, \cF^\infty_n) = \rho(\sigma(X_0), \sigma(X_n));\\
\label{eq2.7.3}
\beta(n) &= \beta(\cF^0_{-\infty},\cF^\infty_n) = \beta(\sigma(X_0), \sigma(X_n)).
\end{align}
(See e.g.\   [\cite{ref-journal-Bradley2007}, vol. 1, Theorem 7.5].)
\medskip

(B) By strict stationarity and (\ref{eq2.7.1})--(\ref{eq2.7.3}), one has that for any integer $j$ and any positive integer $n$, the (strictly stationary) Markov chain $X$ satisfies $\alpha(n) = \alpha(\sigma(X_j),\sigma(X_{j+n}))$, $\rho(n) = \rho(\sigma(X_j), \sigma(X_{j+n}))$, and $\beta(n) = \beta(\sigma(X_j), \sigma(X_{j+n}))$.
\medskip

(C) By (B) above, for any pair of positive integers $m$ and $n$, by the Markov property and three applications of (\ref{eq2.4.3}), the given (strictly stationary) Markov chain $X$ satisfies the well known inequality
\begin{equation}
\label{eq2.7.4}
\rho(m+n) \le \rho(m) \cdot \rho(n).
\end{equation}
In particular, for any two positive integers $m$ and $n$, $\rho(m(n+1)) \le \rho(mn) \cdot \rho(m)$. Hence by induction, for every positive integer $m$, one has that the given (strictly stationary) Markov chain $X$ satisfies
\begin{equation}
\label{eq2.7.5}
\rho(mn) \le [\rho(m)]^n\,\,\mbox{for every $n\in\N$.}
\end{equation}
In particular (take $m=1$),
\begin{equation}\label{eq2.7.6}
\rho(n) \le [\rho(1)]^n\,\,\mbox{for every $n\in \N$.}
\end{equation}

(D) By (\ref{eq2.7.5}), for the given (strictly stationary) Markov chain $X:= (X_k, k\in \Z)$, the following three conditions are equivalent:

(i) there exists $m\ge 1$ such that $\rho(m) <1$;

(ii) $X$ is $\rho$-mixing;

(iii) $\rho(n) \to 0$ at least exponentially fast as $n\to \infty$.
\end{remark}

\begin{example} \label{ex2.8} In (\ref{eq2.7.4}) and (\ref{eq2.7.5}), equality sometimes fails to  hold. Here is a quick review of a well known strictly stationary, 4-state, 1-dependent Markov chain, satisfying
\begin{equation}
\label{eq2.8.1}
\rho(1)=1 \quad\mbox{and}\quad \rho(2) =0.
\end{equation}

Let $X:= (X_k,k\in \Z)$ be a strictly stationary Markov chain with state space $S:= \{1,2,3,4\}$, with (invariant) marginal distribution $\mu$ of $X_0$ given by $\mu(\{s\}) = P(X_0 =s) = 1/4$ for $s\in S$, and with one-step transition probabilities $p(i,j) = P(X_1 = j|X_0=i)$ for $i,j\in S$ (compatible with the marginal distribution $\mu$) given by
\begin{equation}
\label{eq2.8.2}
\begin{array}{c}
p(i,j) = 1/2\,\,\mbox{for}\,\,(i,j) \in \{(1,1),(1,2),(2,3), (2,4), (3,1), (3,2),(4,3),(4,4)\}
\quad\mbox{and}\\
\quad\\
p(i,j) =0\,\,\mbox{for all other ordered pairs}\,\,(i,j) \in S\times S.
\end{array}
\end{equation}

By a simple argument, the events $\{X_0=1$ or $3\}$ and $\{X_1=1$ or $2\}$ are identical modulo sets of probability $0$, and their probability is $1/2$. Hence by a trivial calculation,
\[
\corr(I_{\{1,3\}}(X_0),I_{\{1,2\}}(X_1))=1.
\]
Hence the first equality in (\ref{eq2.8.1}) holds.

With (matrix) multiplication of the $4\times 4$ one-step transition probability matrix (given implicitly) in (\ref{eq2.8.2}) with itself, one sees that the random variable $X_2$ is independent of $X_0$. Hence by (\ref{eq2.7.2}), the second equality in (\ref{eq2.8.1}) holds.

This Markov chain $X$ is not ``reversible'' --- note that since $p(2,1)=0$ and $p(1,2) = 1/2$ by (\ref{eq2.8.2}), it follows that $P(\{X_0 =2\}\cap \{X_1 =1\})=0$ but $P(\{X_0=2\} \cap \{X_{-1} =1\}) = 1/8 \not= 0.$ However, the distributions of the random vectors $(X_0,X_m)$ and $(X_0, X_{-m})$ (or by stationarity $(X_m,X_0))$ for $m\ge2$ (but not $m=1$) are all the same --- the product measure $\mu\times\mu$ where $\mu$ is the marginal distribution.
\end{example}

\section{Strictly stationary, reversible Markov chains}
\label{sc3}
In this section, there is absolutely no assumption of ``irreducibility'' (in the sense related to ``Harris recurrence'').

This section will start with a formal repetition of the definition (in the relevant context) of the term ``reversible''.

\begin{definition}
\label{df3.1}
Suppose $X :=(X_k$, $k\in \Z)$ is a strictly stationary Markov chain with state space $(S,\cS)$ (a measurable space). This Markov chain $X$ is ``reversible'' if the distribution (on $(S^\Z, \cS^\Z))$ of the ``time-reversed'' Markov chain $(X_{-k}, k\in\Z)$ is the same as that of the Markov chain $X$ itself.
\end{definition}

\begin{remark}
\label{rm3.2}
(A) It is well known that a given strictly stationary Markov chain $X:= (X_k$, $k\in \Z)$ with state space $(S,\cS)$ is reversible if and only if the distribution (on $(S^2,\cS^2))$ of the random vector $(X_0,X_1)$ is the same as that of the random vector $(X_1,X_0)$.

(B) Of course, if a given strictly stationary Markov chain $X:=(X_k,k \in \Z)$ is reversible, then for each positive integer $m$, the distribution (on $(S^2,\cS^2))$ of the random vector $(X_0,X_m)$ is the same as that of the random vector $(X_m,X_0)$.

(C) For $m\ge 2$, a ``converse'' of (B) is false.  See Example \ref{ex2.8}.

(D) The main ``point of reference'' for the entire exposition in this paper is the next theorem, practically all of it well known, at least in spirit.
\end{remark}

\begin{theorem}
\label{th3.3}
Suppose $X:= (X_k$, $k\in\Z)$ is a strictly stationary, reversible Markov chain with state space $(S,\cS)$.
Let $\mu$ denote the (marginal) distribution (on $(S, \cS))$ of the random variable $X_0$.

Then the following seven conditions (R1), (R2), (R3), (R4), (A1), (A2), (A3) are equivalent:

\noindent (R1) $\rho(X,1)<1$ (the ``$\cL^2$ spectral gap'' condition).

\noindent (R2) The Markov chain $X$ is $\rho$-mixing

\noindent (R3) $\rho(X,n) \to 0$ at least exponentially fast as $n\to \infty$.

\noindent (R4) There exists a number $r\in [0,1)$ such that $\rho(X,n) = r^n$ for all $n\in \N$.

\noindent (A1) $\alpha(X,n) \to 0$ at least exponentially fast as $n\to \infty$.

\noindent (A2) There exists a number $r\in(0,1)$ such that for every pair of sets $A\in \cS$ and $B\in \cS$, one has (see (\ref{eq2.1.1})) that
\begin{equation}
\label{eq3.3.1}
|P(\{X_0 \in A\} \cap \{X_{\doub(n)} \in B\}) - \mu(A)\mu(B)| \ll r^{\doub(n)}\,\,\mbox{as $n\to \infty$.}
\end{equation}

\noindent (A3) For every set $A\in \cS$, there exists a number $c_A\in (0,1)$ such that (see (\ref{eq2.1.1}))
\begin{equation}
\label{eq3.3.2}
|P(\{X_0 \in A\} \cap \{X_{\doub(n)} \in A\})
 - [\mu(A)]^2| \ll c_A^{\doub(n)}\,\,\mbox{as $n\to \infty$.}
\end{equation}
\end{theorem}

\begin{remark}
\label{rm3.4}
(A) In Theorem \ref{th3.3}, the inclusion of both conditions (R3) and (R4) is clumsy, but will help facilitate a distinction between ``trivial implications that do not involve reversibility'' (in eq.\ (\ref{eq3.4.1}) below) and ``nontrivial implications that do involve reversibility'' (in Proposition \ref{pr3.5}(I)(II)(III) below).

(B) In Theorem \ref{th3.3}, in the labels ((R3),(A1), etc.) for the seven conditions, the letters R and A are intended to ``match'' respectively the Greek letters $\rho$ and $\alpha$ (for the $\rho$-mixing and $\alpha$-mixing conditions) involved (directly or indirectly) in those conditions. (In Remark \ref{rm4.4} in the next section, a similar convention is followed with respect to the letters B and $\beta$ (in connection with the $\beta$-mixing condition).

(C) For a given strictly stationary Markov chain $X:=(X_k, k\in \Z)$ (reversible or not), one trivially has (see (\ref{eq2.6.1}) and Remark \ref{rm2.7}(D)) the implications
\begin{equation}
\label{eq3.4.1}
(R1) \Longrightarrow [(R2) \Longleftrightarrow (R3)] \Longrightarrow (A1) \Longrightarrow (A2) \Longrightarrow (A3).
\end{equation}
(In (\ref{eq3.4.1}), the brackets are intended to highlight the ``trivial equivalence'' of conditions (R2) and (R3).)

(D) As explained in Remark \ref{rm3.6}(A) below, one has that eq.\ (\ref{eq3.4.1}) above and Proposition \ref{pr3.5} below together imply all of Theorem \ref{th3.3}.
\end{remark}

\begin{proposition}\label{pr3.5}
Under the hypothesis (the entire first paragraph, \textit{including reversibility)}
of Theorem \ref{th3.3}, the following three statements hold:

(I) In the notations of Theorem \ref{th3.3}, if condition (A2) holds, then condition (R1) (the ``$\cL^2$ spectral gap'' condition) holds.

(II) In the notations of Theorem \ref{th3.3}, if condition (A3) holds, then condition (A2) holds.

(III) (cf.\ Longla \cite{ref-journal-Longla14}, Lemma 2.1).   Regardless of the value (in $[0,1]$) of $\rho(X,1)$, one has that for every $n\in \N$, $\rho(X,n)=[\rho(X,1)]^n$.
\end{proposition}

Proposition \ref{pr3.5}(I) and its proof (as reviewed in Section \ref{sc5}) are well known, at least implicitly. In ``Theorem \ref{th1.1}.1'', in certain natural proofs that ``geometric ergodicity'' implies the ``$\cL^2$ spectral gap'' condition (R1) (in Theorem \ref{th3.3}), once certain minor ``frills'' are stripped away, the argument just boils down to a (by now routine) proof that condition (A2) (or some ``close cousin'' of it) implies (R1).

Proposition \ref{pr3.5}(II) is a version, or at least a ``close cousin'', of certain formulations already in the literature. In particular, condition (A3) (in Proposition \ref{pr3.5}(II) and Theorem \ref{th3.3}) is intended to serve as a ``quasi-analog'' of condition (B3) in Remark \ref{rm4.4} in the next section. The proof of Proposition \ref{pr3.5}(II) that is reviewed in Section \ref{sc7}, involves just routine elementary arguments.

Proposition \ref{pr3.5}(III) is taken essentially from Longla [\cite{ref-journal-Longla14}, Lemma 2.1]. Its formulation there was formally in a context involving certain types of copulas, but it easily extends beyond that context.  The proof of it that is reviewed in Section \ref{sc6}, has a spirit somewhat like that of the argument given by Longla, but differs in some technical details in order to take advantage of material in Section \ref{sc5} (in the review of a proof of Proposition \ref{pr3.5}(I)).

As noted above, proofs of Statements (I), (II), and (III) in Proposition \ref{pr3.5} will be reviewed respectively in (note the ordering) Sections \ref{sc5}, \ref{sc7}, and \ref{sc6}.

\begin{remark}
\label{rm3.6}
(A) In the context of Theorem \ref{th3.3} (including reversibility), the following three comments hold:

\noindent (i) Eq.\ (\ref{eq3.4.1}) and Proposition \ref{pr3.5}(I) together imply the equivalence of conditions (R1), (R2), (R3), (A1), and (A2).

\noindent (ii) Then eq.\ (\ref{eq3.4.1}) (its last ``implication'') and Proposition \ref{pr3.5}(II)  together imply that condition (A3) is equivalent to (A2)---and hence to the other conditions in (i) above (that is, to the other conditions in (\ref{eq3.4.1})).

\noindent (iii) Then as a trivial by-product of Proposition \ref{pr3.5}(III), condition (R4) is equivalent to (R1) --- and hence to all of the other conditions in the conclusion of Theorem \ref{th3.3}.

\textit{ Thus (\ref{eq3.4.1}) and Proposition \ref{pr3.5} together imply all of Theorem \ref{th3.3}.}

(B) In Proposition \ref{pr3.5}, at least in Statements (I) and (III), the assumption of ``reversibility'' cannot be altogether omitted. With regard to Statement (I), see Remark \ref{rm4.9} and Remark \ref{rm4.5}(B) in the next section. With regard to Statement (III), recall the ``non-reversible'' strictly stationary Markov chain in Example \ref{ex2.8}, satisfying $\rho(1) =1$ and $\rho(n) =0 \not= [\rho(1)]^n$ for all $n\ge 2$.

(It does not seem clear whether or not Statement (II) in Proposition \ref{pr3.5} would still hold if the assumption of ``reversibility'' is omitted. The proof of Statement (II) given in Section \ref{sc7} employs the assumption of reversibility.)

(C) Corollary \ref{co5.7} at the end of Section \ref{sc5} will provide some ``matching of rates'' information with regard to the equivalent conditions in Theorem \ref{th3.3}. It is an ``analog'' of, and is in spirit ``contained'' in, a ``matching of rates'' result from Roberts and Tweedie [\cite{ref-journal-RT}, Theorem 3], building on Roberts and Rosenthal [\cite{ref-journal-RR}, proof of Theorem 2.1], with regard to ``Theorem \ref{th1.1}.1''.
\end{remark}

\begin{remark}\label{rm3.7}  If a given strictly stationary (not necessarily reversible) Markov chain $X:= (X_k, k\in \Z)$ has finite or countably infinite state space and is strongly mixing (or $\rho$-mixing or absolutely regular), then trivially it is irreducible. With that kept in mind, the following comments below hold. Those comments will be given with some frivolous redundancy (in light of Theorem \ref{th3.3} itself, Corollary \ref{co4.6} in the next section, and some inequalities in Section \ref{sc2}).

(A) In Theorem \ref{th3.3} above and in Corollary \ref{co4.6} in the next section, in the absence of further assumptions, there is no role for the ``$\rho^*$-mixing'' condition---the stronger variant of $\rho$-mixing in which the two index sets are allowed to be ``interlaced'' instead of being restricted to ``past'' and ``future''. The author \cite{ref-journal-Bradley2015} constructed a class of strictly stationary, countable-state, reversible Markov chains $X := (X_k, k\in \Z)$ that

\noindent (i) satisfy strong mixing, $\rho$-mixing, and absolute regularity but

\noindent (ii) fail to satisfy $\rho^*$-mixing.

With regard to (i), in those examples the mixing rates can all be made to be ``bounded above by an arbitrarily fast exponential decay'' (but cannot be ``strictly faster than exponential''). With regard to (ii), those examples satisfy $\rho(\sigma(X_0), \sigma(X_{-n},X_n)) =1$ for every positive integer $n$.

(B) It has long been well known that even for strictly stationary Markov chains that are reversible and irreducible and satisfy $\rho$-mixing, Doeblin's condition does not necessarily hold.

Thus in Theorem \ref{th3.3} above and in Corollary \ref{co4.6} in the next section, in the absence of further assumptions, there is no role for Doeblin's condition. For example the (strictly stationary, countable-state, reversible, $\rho$-mixing, etc.) Markov chains $X:= (X_k, k\in \Z)$ in \cite{ref-journal-Bradley2015} alluded to in (A) above can be constructed in such a way that, without any changes in the properties stated in \cite{ref-journal-Bradley2015} (or in (A) above), the following ``extra property'' holds in an ``elementary, transparent'' manner:

For every positive integer $\ell$ and every $\delta >0$, there exist subsets $A$ and $B$ of the (countably infinite) state space, with $P(X_0\in A)>0$, such that
\begin{equation}\label{eq3.4}
P(X_\ell \in B)<\delta\quad\mbox{and}\quad P(X_\ell \in B \mid X_0 \in A) > 1-\delta.
\end{equation}

\textbf{Justification for (B).} The justification of the assertion in (B) will involve just trivial ``bookkeeping'', but plenty of it. It will be spelled out here.

First note that in Lemma 2.2 in \cite{ref-journal-Bradley2015}, for the given integer $N\ge 3$ in that lemma, one also has the following extra properties: First, for every $\varepsilon \in (0,1/3]$ and every $i \in \{0,1,2,\dots, N\}$,
\begin{equation}\label{eq3.5}
P(Y_0^{(\varepsilon)} = i)>0.
\end{equation}
Second, there exists a number $\varepsilon'_N \in (0,1/3]$ such that for every $\varepsilon \in (0, \varepsilon'_N]$ and every $\ell \in \{0,1,2,\dots,N-1\}$,
\begin{equation}\label{eq3.6}
P(Y_0^{(\varepsilon)} = N-\ell) < 1/N.
\end{equation}
Third, there exists a number $\varepsilon''_N \in (0,1/3]$ such that for every $\varepsilon \in (0, \varepsilon''_N]$ and every $\ell \in \{1,2,\dots, N-1\}$,
\begin{equation}\label{eq3.7}
P(Y_\ell^{(\varepsilon)} = N - \ell\mid Y^{(\varepsilon)}_0 = N) >1 - 1/N.
\end{equation}

To verify the sentence containing (\ref{eq3.5}) above, apply in \cite{ref-journal-Bradley2015} its eq. (2.1) and all four lines of its eq.\ (2.2) in its Definition 2.1. To verify the sentence containing (\ref{eq3.6}) above, note that in \cite{ref-journal-Bradley2015}, in its Definition 2.1, by the second, third, and fourth lines of its eq.\ (2.2),
\[
\lim_{\varepsilon \to 0+}\left[ \max_{\ell \in \{0,1,\dots, N-1\}} \mu_{N,\varepsilon, N-\ell}\right] = \lim_{\varepsilon \to 0+} \left[ \max_{m\in\{1,2,\dots,N\}} \mu_{N,\varepsilon, m}\right] =0.
\]
To verify the sentence containing (\ref{eq3.7}), above, note that in \cite{ref-journal-Bradley2015}, in its Definition 2.1, by the fourth line of its eq.\ (2.6),
\[
\lim_{\varepsilon \to 0+} \left[ \min_{\ell\in \{1,2,\dots,N-1\}} p^{(\ell)}_{N,\varepsilon, N,N-\ell}\right] = \lim_{\varepsilon \to 0+}\left[\min_{i \in \{1,2,\dots,N-1\}} p^{(N-i)}_{N,\varepsilon, N,i}\right] =1.
\]

Next, in \cite{ref-journal-Bradley2015}, applying its Lemma 2.2 as augmented in the sentences containing eqs.\ (\ref{eq3.5}), (\ref{eq3.6}), and (\ref{eq3.7}) above, choose the strictly stationary Markov chains $Z^{(N)}:= (Z^{(N)}_k, k\in \Z)$ for $N\in \{3,4,5,\dots\}$ on page 84 (of \cite{ref-journal-Bradley2015}) in such a way as to have, in addition to all properties stated in lines 6-15 of page 84 (of \cite{ref-journal-Bradley2015}), the following extra properties:

\noindent(i) $P(Z_0^{(N)} =i)>0$ for every $i\in \{0,1,2,\dots,N\}$;

\noindent(ii) $P(Z^{(N)}_0 = N-\ell) <1/N$ for every $\ell \in \{0,1,\dots,N-1\}$; and

\noindent(iii) $P(Z^{(N)}_\ell = N-\ell \mid Z^{(N)}_0 = N)>1 - 1/N$ for every $\ell \in \{1,2,\dots,N-1\}$.
\medskip

Then for every positive integer $\ell$ and every $\delta>0$, one has for any given integer $N>\max\{3,\ell, 1/\delta\}$ that
\[
P(Z^{(N)}_\ell = N-\ell) = P(Z^{(N)}_0 =N-\ell)< 1/N <\delta\,\,\,\mbox{and}
\]
\[
P(Z^{(N)}_\ell = N-\ell \mid Z^{(N)}_0 =0) >1 - 1/N >1-\delta.
\]
Then from eq.\ (3.8) in \cite{ref-journal-Bradley2015} one completes
 the ``elementary, transparent'' verification of the ``extra property'' stated in the sentence containing (\ref{eq3.4}) above.
\end{remark}

\section{Strictly stationary, irreducible Markov chains}\label{sc4}
In this section, there is no assumption of ``reversibility'', except when that assumption is stated explicitly.

Also in this section, to bypass some slight technicalities, the (strictly stationary) Markov chains are assumed to have state space $(R, \cR)$.

The main ``setting'' for this section is as follows:

\begin{context} \label{ctxt4.1}  Suppose $X:= (X_k$, $k\in\Z)$ is a (not necessarily reversible) strictly stationary Markov chain with state space $(R,\cR)$.  Let $\mu$ denote the (marginal) distribution (on $(R,\cR))$ of the random variable $X_0$. Abusing notations slightly, let $P((X_1,X_2,X_3,\dots) \in D\mid X_0 =x)$, for $x\in \R$ and $D\in \cR^\N$, denote a regular conditional distribution of the (``one-sided'') random sequence $(X_1,X_2,X_3,\dots)$ given $X_0$.
\end{context}

(By certain measure-theoretic arguments that will not be reviewed here, it will not matter at all which particular ``regular conditional distribution'' is chosen.)

\begin{definition}
\label{df4.2}
(A) in Context \ref{ctxt4.1}, the (strictly stationary) Markov chain $X$ is ``irreducible'' if there exists a set $A \in \cR$ satisfying $\mu(A) =1$ such that the following holds: For every $x\in A$ and every set $B\in \cR$ such that $\mu(B)>0$, there exists a positive integer $n$ (depending on $x$ and $B$) such that $P(X_n \in B\mid X_0 =x)>0$.

(B) In Context \ref{ctxt4.1}, the (strictly stationary) Markov chain $X$ is ``Harris recurrent'' if there exists a set $A\in \cR$ satisfying $\mu(A) =1$ such that the following holds: For every $x\in A$ and every set $B\in \cR$ such that $\mu(B) >0$, one has that
\[
P(X_n\in B\,\,\mbox{for infinitely many }\,\,n\in \N \mid X_0 =x)=1.
\]
\end{definition}

\begin{remark}\label{rm4.3}
The following information can be found (in different terminologies) in Orey \cite{ref-journal-Orey}, Meyn and Tweedie \cite{ref-journal-MT}, and many other references. A generously detailed exposition of it in the terminology employed here, can be found in [\cite{ref-journal-Bradley2007}, vol.\ 2, Chapter 21].

(A) In Context \ref{ctxt4.1}, for the given (strictly stationary) Markov chain $X:= (X_k$, $k\in \Z)$, the following three conditions (i), (ii), (iii) are equivalent:

\noindent (i) $X$ is irreducible (in the sense of Definition \ref{df4.2}(A));

\noindent (ii) $X$ is Harris recurrent;

\noindent (iii) $X$ is ergodic, and there exists $n\in \N$ such that $\beta(X,n) <1$.

(B) In Context \ref{ctxt4.1}, for the given (strictly stationary) Markov chain $X:= (X_k$, $k\in \Z)$, if $X$ is irreducible (in the sense of Definition \ref{df4.2}(A)) --- that is, if $X$ satisfies the three equivalent conditions (i), (ii), and (iii) in (A) above --- then $X$ has what Orey \cite{ref-journal-Orey}  calls a ``$C$-set'' and Meyn and Tweedie \cite{ref-journal-MT}  and some other references call a ``small set'': a set $C\in \cR$, accompanied by a positive number $t$ and a positive integer $n$, such that

\noindent (i) $\mu(C)>0$, and

\noindent (ii) for every pair of sets $A$, $B\in \cR$ such that $A\subset C$ and $B\subset C$, one has that
\[
P(\{X_0\in A\} \cap \{X_n \in B\}) \ge t\cdot \mu (A) \cdot \mu(B).
\]
\end{remark}

\begin{remark}\label{rm4.4}
In Context \ref{ctxt4.1} (with no assumption of ``reversibility''), the following four conditions (B1), (B2), (B3), and (B4) are equivalent:

\noindent (B1) (``geometric ergodicity'') There exist a set $A\in \cR$ such that $\mu(A) =1$, and Borel functions $G: A\to [0,\infty)$ and $\theta: A\to (0,1)$ such that for every $x\in A$ and every $n\in \N$,
\[
\left[\sup_{B\in \cR} |P(X_n\in B\mid X_0 =x) - \mu(B)|\right] \le G(x) \cdot [\theta(x)]^n.
\]

\noindent (B2) There exist a set $A\in \cR$ such that $\mu(A) =1$, a Borel function $G: A\to [0,\infty)$, and a number $\lambda \in (0,1)$ such that for every $x\in A$ and every $n\in \N$,
\[
\left[ \sup_{B\in \cR} |P(X_n \in B\mid X_0 =x) - \mu (B)|\right] \le G(x) \cdot \lambda^n.
\]

\noindent (B3) There exists (for the Markov chain $X$) a ``$C$-set'' (that is, a ``small set'') $C$ such that
\begin{equation}\label{eq4.4.1}
\left|P(X_n \in C \mid X_0 \in C) - \mu(C)\right| \longrightarrow 0\,\,\mbox{ at least exponentially fast as $n\to \infty$.}
\end{equation}

\noindent (B4) $\beta(X,n) \to 0$ at least exponentially fast as $n\to \infty$.
\end{remark}

The equivalence of (B1), (B2), and (B3) was shown by Nummelin and Tweedie [\cite{ref-journal-NTweed}, Theorem 1]. The equivalence of (B4) with (B1)--(B2)--(B3) was shown by Nummelin and Tuominen [\cite{ref-journal-NTuo}, Theorem 2.1]. (There condition (B4) was formulated in an equivalent form with different terminology.)  Both papers built on insights from earlier papers such as Kendall \cite{ref-journal-Kendall}  and Vere-Jones \cite{ref-journal-VJ}. A generously detailed presentation of the equivalence of those four conditions is given by  [\cite{ref-journal-Bradley2007}, vol.\ 2, Theorems 21.13 and 21.19]. That material will be taken for granted, and not reviewed in detail here in this expository paper.

\begin{remark}\label{rm4.5} (A) In Context \ref{ctxt4.1} (with no assumption of ``reversibility''), the four equivalent conditions (B1)--(B2)--(B3)--(B4) in Remark \ref{rm4.4} imply that the given (strictly stationary) Markov chain $X$ is ``irreducible'' (in the sense of Definition \ref{df4.2}(A)).  (Compare condition (B4) in Remark \ref{rm4.4} with condition (iii) in Remark \ref{rm4.3}(A).)

(B) In Context \ref{ctxt4.1} (with no assumption of ``reversibility''), by (\ref{eq2.6.2}), condition (B4) in Remark \ref{rm4.4} trivially implies condition (A1) (and hence trivially also conditions (A2) and (A3)) in Theorem \ref{th3.3}.

(C) In Context \ref{ctxt4.1}, if the given strictly stationary (not necessarily reversible) Markov chain $X:= (X_k$, $k\in \Z)$ is irreducible (in the sense of Definition \ref{df4.2}(A)) and satisfies condition (R3) in Theorem \ref{th3.3}, then it satisfies condition (B3) (and hence also conditions (B1), (B2), and (B4)) in Remark \ref{rm4.4} as well. The argument is well known and short (see e.g.\ \cite{ref-journal-KM}), and it will be reviewed here:

Applying Remark \ref{rm4.3}(B) and the assumption here of irreducibility, let $C\in \cR$ be a $C$-set (that is, a ``small set'') for the Markov chain $X$. By the assumption here of condition (R3), one has that $|\corr(I_C(X_0), I_C(X_n))|$ $\to 0$ at least exponentially fast as $n\to \infty$. Hence by trivial arithmetic, the same holds for $|\cov (I_C(X_0), \allowbreak I_C(X_n))|$; and by slight further trivial arithmetic, eq.\ (\ref{eq4.4.1}) holds.  Thus condition (B3) holds. That completes the argument.
\medskip

 The next statement, called a ``corollary'', follows from Theorem \ref{th3.3}, Remark \ref{rm4.4}, and Remark \ref{rm4.5}(B)(C).  It is primarily just a list of conditions that are known to be equivalent under certain hypotheses.
\end{remark}

\begin{corollary}\label{co4.6}
In Context \ref{ctxt4.1}, if the given (strictly stationary) Markov chain $X := (X_k, k\in \Z)$ is both reversible and irreducible (in the sense of Definition \ref{df4.2}(A)), then the following eleven conditions are equivalent:

\noindent conditions (R1), (R2), (R3), (R4) in Theorem \ref{th3.3};

\noindent conditions (A1), (A2), (A3) in Theorem \ref{th3.3};

\noindent conditions (B1), (B2), (B3), (B4) in Remark \ref{rm4.4}.
\end{corollary}

\begin{remark}\label{rm4.7} \textrm{
''Theorem \ref{th1.1}.1'' (the limited form of it treated in this expository paper) can now be formulated precisely as a particular part of Corollary \ref{co4.6}; it simply says that under the hypothesis of Corollary \ref{co4.6}, the ``geometric ergodicity'' condition (B1) and the ``$\cL^2$ spectral gap'' condition (R1) are equivalent.}

\end{remark}

\begin{remark}\label{rm4.8}
In Corollary \ref{co4.6}, if the assumption of ``irreducibility'' were omitted, then conditions (B1), (B2), (B3), and (B4) would have to be omitted, and one would be left with Theorem \ref{th3.3}. Let us elaborate on that a little.

(A) If a given strictly stationary Markov chain $X:= (X_k, k\in \Z)$ (reversible or not) satisfies $\beta(X,n) =1$ for all $n\in \N$, then it fails to be irreducible. (In Remark \ref{rm4.3}(A), look at condition (iii).)

(B) With the use of (in effect) ``random rotations'', Rosenblatt [\cite{ref-journal-Rosenblatt1971}, p.\ 214, line $-3$ to p.\ 215, line 13] constructed a class of strictly stationary Markov chains $X:= (X_k, k\in \Z)$ that are $\rho$-mixing but satisfy $\beta(X,n) =1$ for all $n\in \N$. It seems clear that some of those examples (the ones whose underlying ``random rotations'' satisfy a certain symmetry condition) are also reversible.

(C) A pair of examples --- in disguise, slight variations on those of Rosenblatt \cite{ref-journal-Rosenblatt1971} cited in (B) above --- was constructed in [\cite{ref-journal-Bradley2007}, vol.\ 1, Examples 7.16 and 7.18]. Those two examples, strictly stationary Markov chains indexed by $\Z$, are reversible and are $\rho$-mixing (and even satisfy $\rho^*$-mixing --- the stronger variant of $\rho$-mixing alluded to in Remark \ref{rm3.7}(A)), and they also satisfy $\beta(n) =1$ for all $n\in \N$.
\end{remark}

[In those two examples, the property of ``reversibility'' was not discussed, but it is easy to see as inherited from the ``building blocks'' (of those examples) --- certain strictly stationary, 2-state Markov chains that are themselves easily seen to be reversible.]

\begin{remark}\label{rm4.9}
In Corollary \ref{co4.6}, if the assumptions of ``reversibility'' were omitted, then parts of the conclusion of that statement would in general become false. Kontoyiannis and Meyn [\cite{ref-journal-KM}, Theorem 1.4 and Proposition 3.1] discussed two known examples of strictly stationary, countable-state, irreducible, non-reversible Markov chains $X:= (X_k, k\in \Z)$ that

\noindent (i) satisfy the conditions (B1)--(B2)--(B3)--(B4) in Remark \ref{rm4.4}, but

\noindent (ii) \textbf{fail} to be $\rho$-mixing (i.e. fail to satisfy conditions (R1), (R2), (R3), and (R4) in Theorem \ref{th3.3}).

\noindent One of those examples was the underlying Markov chain in a construction of H\"aggstr\"om [\cite{ref-journal-Hagg}, Theorem 1.3]. The other example --- a much more complicated one that (refer again to condition (B4) in Remark \ref{rm4.4}) allows $\beta(n)$ to converge to 0 arbitrarily fast --- is the underlying Markov chain in a construction of the author [\cite{ref-journal-Bradley1983}, Theorem 2 and (p.\ 95) Remark 2.1]. (See also the comment by the author [\cite{ref-journal-Bradley1997}, p.\ 545, lines 20-25]). Alternatively, see [\cite{ref-journal-Bradley2007}, vol.\ 3, Theorem 31.3 and Corollary 31.5.]
\end{remark}

\begin{remark}\label{rm4.10}
In the context of Corollary \ref{co4.6} (see Remark \ref{rm4.4}), for a given $r\in (0,1)$, if condition (B4) holds with $\beta(X,n) \ll r^n$ as $n\to \infty$, then condition (R1) holds with $\rho(X,1) \le r$, by eq.\ (\ref{eq2.6.2}) and Corollary \ref{co5.7} in Section \ref{sc5}.

In essence that is ``half'' of a ``matching of rates'' result of Roberts and Tweedie [\cite{ref-journal-RT}, Theorem 3], building on Roberts and Rosenthal [\cite{ref-journal-RR}, proof of Theorem 2.1], in connection with ``Theorem \ref{th1.1}.1''/Remark \ref{rm4.7}. (The other ``half'' of that ``matching of rates'' result will not be reviewed in this expository paper here.)
\end{remark}

\begin{remark}\label{rm4.11}
In Theorem \ref{th3.3} and Corollary \ref{co4.6}, the majority of the stated equivalent conditions involve (at least indirectly) measures of dependence associated with strong mixing conditions.

In the context of either theorem, there are of course many other conditions that are closely related to, and in some cases equivalent to, the stated conditions. See for examples Roberts and Rosenthal \cite{ref-journal-RR}, Roberts and Tweedie \cite{ref-journal-RT}, and Kotoyiannis and Meyn \cite{ref-journal-KM}. For some such related ``other conditions'' that are based on ``measures of dependence'' other than ones associated with ``strong mixing conditions'', see also Beare [\cite{ref-journal-Beare}, Remark 3.4] and the references cited there.
\end{remark}

\section{Review of a proof of Proposition \ref{pr3.5}(I)} \label{sc5}
The following ``Context'' will provide the setting for the work here in Section \ref{sc5}.

\begin{context}\label{ctxt5.1}
(A) Suppose $X:= (X_k, k\in \Z)$ is a strictly stationary Markov chain with state space $S$ (see Convention \ref{cn2.2}(B)(D) again). Suppose this Markov chain $X$ is reversible (but not necessarily ``irreducible'' in the sense of  Definition \ref{df4.2}(A)).

(B) Let $\mu$ denote the (marginal) distribution (on $(S, \cS))$ of the random variable $X_0$.

(C) Let $\cL^1(\mu)$ denote the family of all $\cS/\cR$-measurable functions $g:S\to \R$ such that $\int_{S}|g|d\mu <\infty$. Let $\cL^2 (\mu)$ denote the family of all $\cS/\cR$-measurable functions $g:S\to \R$ such that $\int_{S} g^2d\mu <\infty$.  Of course (since $\mu(S) =1<\infty$) $\cL^2(\mu) \subset \cL^1(\mu)$. Define the following two subclasses of $\cL^2(\mu)$:
\begin{equation}\label{eq5.1.1}
\cL^2_{ub}(\mu) = \left\{ g\in \cL^2(\mu): \int_S g^2d\mu \le 1\right\};
\end{equation}
\begin{equation}\label{eq5.1.2}
\cL^2_{ub0}(\mu) = \left\{ g\in \cL^2(\mu) : \int_{S} g^2d\mu \le 1\,\,\,\mbox{and}\,\,\int_S gd\mu =0\right\}.
\end{equation}
Obviously $\cL^2_{ub0}(\mu) \subset \cL^2_{ub}(\mu) \subset \cL^2(\mu) \subset \cL^1(\mu)$. In (\ref{eq5.1.1}) and (\ref{eq5.1.2}), the letters ``$ub$'' in the subscript stand for ``(closed) unit ball''. In (\ref{eq5.1.2}), the digit 0 in the subscript stands for ``mean $0$'' --- the condition $\int_S gd\mu =0$.

(D) For any $g\in \cL^1(\mu)$ and any integer $k$, one of course has that
\begin{equation}\label{eq5.1.3}
E|g(X_k)| = \int_S |g|d\mu < \infty \,\,\,\,\mbox{and}\,\,\,\, E[g(X_k)] = \int_Sgd\mu.
\end{equation}
For any $g\in \cL^2(\mu)$ and any integer $k$, one of course has that
\begin{equation}
\label{eq5.1.4}
E[(g(X_k))^2] = \int_Sg^2 d\mu.
\end{equation}
\end{context}

The following lemma and some related subsequent material apply to any integer $m$, but for simplicity they will be restricted to just positive integers $m$.
\begin{lemma}\label{lm5.2}
In Context \ref{ctxt5.1}, one has that for  any positive integer $m$ and any $f\in \cL^1(\mu)$,
\begin{equation}
\label{eq5.2.1}
E \left[ f(X_{2m})\mid \sigma(X_m)\right] = E\left[ f(X_0) \mid \sigma(X_m) \right] \,\,\mbox{a.s.}
\end{equation}
\end{lemma}

\begin{proof}
Suppose $m$ and $f$ are as in the statement of Lemma \ref{lm5.2}. Both sides of (\ref{eq5.2.1}) are measurable with respect to $\sigma(X_m)$. Hence to show the a.s.\ equality in (\ref{eq5.2.1}), it suffices to show that for every event $G\in \sigma(X_m)$,
\begin{equation}
\label{eq5.2.2}
\int_G E\left[f(X_{2m} \mid \sigma(X_m)) \right]dP = \int_G E \left[ f(X_0 \mid \sigma(X_m)) \right] dP.
\end{equation}

Let $G\in \sigma(X_m)$ be arbitrary but fixed. Our task now is to prove (\ref{eq5.2.2}) for this fixed $G$.

Now (recall Convention \ref{cn2.2}(E)), there exists a (now fixed) set $A \in \cS$ such that $G = \{X_m \in A\}$. Then of course for the indicator functions $I_G: \Omega \to \{0,1\}$ and $I_A: S \to \{0,1\}$, the $\{0,1\}$-valued random variables $I_G$ and $I_A(X_m)$ are identical. (For a given $\omega \in \Omega$, one has the equivalencies $I_G (\omega) =1 \Longleftrightarrow \omega \in \{X_m \in A\} \Longleftrightarrow I_A(X_m(\omega)) =1.$)

Define the $\cS^2/\cR$-measurable function $\lambda :S^2 \to \R$ as follows: For every element $(t,u)\in S^2$, $\lambda(t,u) = (f(t) \cdot I_A(u))$.

Since $G \in \sigma(X_m)$, one has that
\begin{align}\label{eq5.2.3}
&\int_G E\left[ f(X_0) \mid \sigma(X_m) \right] dP = \int_G f(X_0)\, dP = \int_\Omega f(X_0) \cdot I_G \,dP \nonumber\\
= & \int_\Omega f(X_0) \cdot I_A(X_m) \,dP = \int_\Omega \lambda (X_0, X_m)\, dP = E[\lambda (X_0, X_m)].
\end{align}
By the same argument as in (\ref{eq5.2.3}) but with $X_0$ replaced by $X_{2m}$, one has that
\begin{equation}
\label{eq5.2.4}
\int_G E[f(X_{2m}) \mid \sigma(X_m) ]dP = E[\lambda(X_{2m}, X_m)].
\end{equation}
Since the Markov chain $X$ is strictly stationary and reversible, the distribution (on $(S^2, \cS^2))$ of the random vector $(X_{2m},X_m)$ is the same as that of the random vector $(X_m, X_0)$, which in turn is the same as that of $(X_0,X_m)$. Hence $E[\lambda(X_{2m},X_m)] =E[\lambda(X_0, X_m)]$. Hence by (\ref{eq5.2.3}) and (\ref{eq5.2.4}), eq.\ (\ref{eq5.2.2}) holds. That completes the proof of Lemma \ref{lm5.2}.
\end{proof}

\begin{remark}
\label{rm5.3}
In Context \ref{ctxt5.1}, the following observations hold:

(A) For any positive integer $m$ and any $f\in \cL^2(\mu)$ one has by the Markov property and Lemma \ref{lm5.2} that
\begin{align}\label{eq5.3.1}
E &[f(X_0)\cdot f(X_{2m})] = E\Bigl[E[f(X_0) \cdot f(X_{2m}) \mid \sigma(X_m)]\Bigr] \nonumber\\
&= E\bigl[ E[f(X_0) \mid \sigma(X_m)] \cdot E[f(X_{2m})\mid \sigma(X_m)]\bigr] \nonumber\\
&=E\left[(E[f(X_0)\mid \sigma(X_m)])^2\right] = \| E[f(X_0) \mid \sigma(X_m) ] \|^2_2 \ge 0.
\end{align}

(B) For any positive integer $m$ and any functions $g,h\in \cL^2_{ub}(\mu)$, one has by Cauchy's Inequality and (\ref{eq5.3.1}) that
\begin{align}\label{eq5.3.2}
0 \le &\left| E[g(X_0) \cdot h(X_m) ]\right| =  \left|E[E[g(X_0) \cdot h(X_m)\mid \sigma(X_m)]]\right|\nonumber\\
&= \left| E[h(X_m) \cdot E[g(X_0) \mid \sigma(X_m)]]\right| \le E\left| h(X_m) \cdot E[g(X_0) \mid \sigma(X_m)]\right| \nonumber\\
&\le \| h(X_m) \|_2 \cdot \| E[g(X_0) \mid \sigma(X_m) ] \|_2 \nonumber\\
&\le 1\cdot \| E[g(X_0) \mid \sigma(X_m) ] \|_2 = \left( E[g(X_0) \cdot g(X_{2m}) ] \right)^{1/2}.
\end{align}

(C) For any positive integer $m$ and any $g\in \cL^2_{ub}(\mu)$, one has by (\ref{eq5.3.2}) (with $h=g$ itself) that
\begin{equation}\label{eq5.3.3}
E[ g(X_0) \cdot g(X_{2m})] \ge |E[g(X_0) \cdot g(X_m) ] |^2 \ge 0,
\end{equation}
and of course also by Cauchy's Inequality and stationarity,
\begin{equation}\label{eq5.3.4}
E[g(X_0) \cdot g(X_{2m})] \le \|g(X_0)\|_2 \cdot \|g(X_{2m})\|_2 = \|g(X_0) \|^2_2 \le 1.
\end{equation}

(D) Refer to (\ref{eq2.1.1}). For every positive integer $n$, the integer $2^n$ is even. Hence for any positive integer $n$ and any $g\in \cL^2_{ub}(\mu)$,
\begin{equation}\label{eq5.3.5}
E[g(X_0) \cdot g(X_{\doub(n)})] \ge 0
\end{equation}
by (\ref{eq5.3.1}) (or (\ref{eq5.3.3})) with $m = (1/2)\cdot 2^n$, and
\begin{equation}\label{eq5.3.6}
\left(E[g(X_0) \cdot g(X_{\doub(n)})]\right)^2 \le E\left[ g(X_0) \cdot g(X_{\doub (n+1)})\right] \le 1
\end{equation}
by (\ref{eq5.3.3}) and (\ref{eq5.3.4}), and hence (take all three terms in (\ref{eq5.3.6}) to the power $1/2^{n+1}$),
\begin{equation}\label{eq5.3.7}
\left( E[g(X_0) \cdot g(X_{\doub(n)})]\right)^{1/\doub(n)} \le \left( E[g(X_0) \cdot g(X_{\doub(n+1)})]\right)^{1/\doub(n+1)} \le 1.
\end{equation}

(E) By (\ref{eq5.3.7}) and induction, one has that for any positive integer $n$ and any $g\in \cL^2_{ub}(\mu)$,
\begin{equation}\label{eq5.3.8}
\left(E[g(X_0) \cdot g(X_2)]\right)^{1/2} \le \left( E[g(X_0) \cdot g(X_{\doub(n)})]\right)^{1/\doub(n)}.
\end{equation}
Hence by (\ref{eq5.3.2}) and (\ref{eq5.3.8}), one has that for any positive integer $n$ and any $g, h\in \cL^2_{ub}(\mu)$,
\begin{equation}\label{eq5.3.9}
\left|E[g(X_0) \cdot h(X_1)]\right| \le \left( E[g(X_0) \cdot g(X_2)]\right)^{1/2} \le \left(E[g(X_0) \cdot g(X_{\doub(n)})]\right)^{1/\doub(n)}.
\end{equation}
\end{remark}

\begin{remark}
\label{rm5.4}
In Context \ref{ctxt5.1}, the following remarks hold:

(A) Suppose $k$ is an integer, and $V$ is a real-valued simple random variable that is measurable with respect to the $\sigma$-field $\sigma(X_k)$, such that
\begin{equation}\label{eq5.4.1}
E(V^2) \le 1\quad \mbox{and}\quad EV =0.
\end{equation}
Let $v(1)$, $v(2)$,$\dots,v(n)$ denote the elements in the range of $V$  (with each such element listed exactly once). Then one has the standard representation
\begin{equation}\label{eq5.4.2}
V = \sum^n_{i=1} v(i) \cdot I_{\{V = v(i)\}}.
\end{equation}
[For any given $\omega \in \Omega$, letting $j\in \{1,2,\dots,n\}$ be such that $V(\omega) = v(j)$, one has that $v(j) \cdot I_{\{V=v(j)\}}(\omega) = v(j) \cdot 1 = v(j)$ and $v(i) \cdot I_{\{V = v(i)\}}(\omega) =v(i) \cdot 0 = 0$ for $i \not= j$, and hence $\sum\limits^n_{i=1} v(i) \cdot I_{\{V=v(i)\}} (\omega) = v(j) = V(\omega)$.]

(B) Continuing with (A), for each $i \in \{1,2,\dots,n\}$,  let $A(i) \in \cS$ be such that $\{V=v(i)\} = \{X_k \in  A(i)\}$. For each $i\in \{1,2,\dots,n\}$, one has the equality of random variables
\begin{equation}\label{eq5.4.3}
I_{\{V = v(i)\}} = I_{A(i)}(X_k).
\end{equation}
[For a given $\omega \in \Omega$, $I_{A(i)}(X_k(\omega)) = 1\Longleftrightarrow \omega \in \{X_k \in A(i)\} \Longleftrightarrow I_{\{V=v(i)\}}(\omega) =1$.]

(C) Continuing with (A)--(B), define the $\cS/\cR$-measurable simple function $g:S\to \R$ as follows: For each $s\in S$,
\begin{equation}\label{eq5.4.4}
g(s) := \sum^n_{i=1} v(i) \cdot I_{A(i)}(s).
\end{equation}
Then by (\ref{eq5.4.2}), (\ref{eq5.4.3}), and (\ref{eq5.4.4}), one has the representation
\begin{equation}\label{eq5.4.5}
V = \sum^n_{i=1} v(i) \cdot I_{A(i)}(X_k) = g(X_k).
\end{equation}
By (\ref{eq5.1.4}), (\ref{eq5.4.5}), and (\ref{eq5.4.1}),
\begin{equation}
\label{eq5.4.6}
\int_S g^2d\mu = E[(g(X_k))^2] = E(V^2) \le 1;
\end{equation}
and by (\ref{eq5.1.3}), (\ref{eq5.4.5}), and (\ref{eq5.4.1}),
\begin{equation}\label{eq5.4.7}
\int_S gd\mu = E[g(X_k)] = EV =0.
\end{equation}
Of course by (\ref{eq5.4.6}) and (\ref{eq5.4.7}), $g\in \cL^2_{ub0}(\mu)$.
\end{remark}

\begin{lemma}\label{lm5.5}
In Context \ref{ctxt5.1}, suppose $0<r<1$, and suppose that for every pair of sets $A,B\in \cS$, one has that
\begin{equation}\label{eq5.5.1}
\left| P(\{X_0 \in A\} \cap \{X_{\doub(n)} \in B\}) - \mu(A) \mu(B)\right| \ll r^{\doub(n)} \,\,\mbox{as $n\to\infty$.}
\end{equation}
Then $\rho(X,1) \le r$.
\end{lemma}

\begin{proof}
As in the statement, suppose $0<r<1$, and that (\ref{eq5.5.1}) holds for every pair of sets $A,B\in \cS$. Suppose
\begin{equation}\label{eq5.5.2}
\rho(X,1) >r.
\end{equation}
We shall aim for a contradiction.

Let $t$ be a real number such that
\begin{equation}\label{eq5.5.3}
r< t<\rho(X,1).
\end{equation}

By (\ref{eq2.7.2}) and (\ref{eq5.5.3}),
\[
\rho(\sigma(X_0),\sigma(X_1)) = \rho(X,1)>t.
\]
Accordingly, applying (\ref{eq2.4.4}) to the $\sigma$-fields $\sigma(X_0)$ and $\sigma(X_1)$, let $V$ and $W$ be real-valued simple random variables such that $V$ is measurable with respect to $\sigma(X_0)$, $W$ is measurable with respect to $\sigma(X_1)$, and
\begin{equation}\label{eq5.5.4}
E(V^2) \le 1,\, E(W^2) \le 1,\,\,EV= EW=0, \,\,\mbox{and}\,\, E(VW)>t.
\end{equation}

Applying the procedure in Remark \ref{rm5.4}, let $\ell$ be a positive integer, and $v(1),v(2),\dots,v(\ell)$ be distinct real numbers (the elements of the range of $V$), and $C(1),C(2),\dots,C(\ell)$ be sets $\in \cS$, such that (\`a la (\ref{eq5.4.5}))
\begin{equation}\label{eq5.5.5}
V = \sum^\ell_{i=1} v(i) \cdot I_{C(i)} (X_0);
\end{equation}
and (see (\ref{eq5.4.4})) define the $\cS/\cR$-measurable simple function $g:S\to \R$ as follows: For $s\in S$,
\begin{equation}\label{eq5.5.6}
g(s) := \sum^\ell_{i=1} v(i) \cdot I_{C(i)} (s).
\end{equation}
Then (\`a la (\ref{eq5.4.5})),
\begin{equation}\label{eq5.5.7}
V = g(X_0).
\end{equation}
Of course by (\ref{eq5.1.4}), (\ref{eq5.5.7}), and (\ref{eq5.5.4}),
\begin{equation}\label{eq5.5.8}
\int_S g^2d\mu = E[(g(X_0))^2] = E(V^2) \le 1;
\end{equation}
and by (\ref{eq5.1.3}), (\ref{eq5.5.7}), and (\ref{eq5.5.4}),
\begin{equation}\label{eq5.5.9}
\int_S gd\mu = E[g(X_0)] = EV =0,
\end{equation}
and hence $g\in \cL^2_{ub0}(\mu)$.

Applying the procedure of Remark \ref{rm5.4} again in the same way, let $h:S \to \R$ be an $\cS/\cR$-measurable simple function such
\begin{equation}\label{eq5.5.10}
W = h(X_1).
\end{equation}
Analogously to (\ref{eq5.5.8}) and (\ref{eq5.5.9}), $h\in \cL^2_{ub0}(\mu)$.

By (\ref{eq5.5.4}), (\ref{eq5.5.7}), and (\ref{eq5.5.10}), $E[g(X_0) \cdot h(X_1)]>t$. Hence by (\ref{eq5.3.9}) (its entire sentence), for every positive integer $n$,
\[
t<\left( E[g(X_0) \cdot g(X_{\doub(n)})]\right)^{1/\doub(n)}.
\]
Hence
\begin{equation}\label{eq5.5.11}
\forall\, n\ge 1,\,\,\, E\left[ g(X_0) \cdot g(X_{\doub(n)})\right] >t^{\doub(n)}.
\end{equation}
Also, since $0<r<t$ by (\ref{eq5.5.3}) (and the hypothesis $r>0$), one has that $t/r >1$ and hence $t^{\doub(n)}/r^{\doub(n)} = (t/r)^{\doub(n)} \longrightarrow \infty$ as $n\to \infty$.  Hence by (\ref{eq5.5.11}),
\begin{equation}
\label{eq5.5.12}
E\left[ g(X_0) \cdot g(X_{\doub(n)})\right]/r^{\doub(n)} \longrightarrow \infty \,\,\,\mbox{as $n\to \infty$.}
\end{equation}
We shall return to (\ref{eq5.5.12}) below. But first we need to develop another, separate line of argument.

Refer to (\ref{eq5.5.5}), (\ref{eq5.5.6}), and (\ref{eq5.5.7}), with the positive integer $\ell$, the (distinct) real numbers $v(1)$, $v(2),\dots, v(\ell)$, and the sets $C(1)$, $C(2),\dots,C(\ell) \in \cS$. If $i$ and $j$ (equal or distinct) are each an element of $\{1,2,\dots, \ell\}$, then for any positive integer $m$,
\[
I_{C(i)}(X_0) \cdot I_{C(j)} (X_m) = I(X_0 \in C(i)) \cdot I(X_m\in C(j))=I(\{X_0 \in C(i)\} \cap \{X_m \in C(j)\})
\]
(where the second and third terms involve indicator functions on $\Omega$), and hence
\[
E\left[ I_{C(i)}(X_0) \cdot I_{C(j)} (X_m)\right] = P( \{X_0 \in C(i)\} \cap \{X_m \in C(j)\}).
\]

Hence for any positive integer $m$, by (\ref{eq5.5.6}),
\begin{align}\label{eq5.5.13}
E[g(X_0) \cdot g(X_m)] &= E\left[\left(\sum^\ell_{i=1} v(i) \cdot I_{C(i)} (X_0)\right)\left(\sum^\ell_{j=1} v(j) I_{C(j)}(X_m)\right)\right] \nonumber\\
&= E\left[ \sum^\ell_{i=1}\sum^\ell_{j=1} v(i) \cdot v(j) \cdot I_{C(i)} (X_0) \cdot I_{C(j)}(X_m)\right] \nonumber\\
&= \sum^\ell_{i=1} \sum^\ell_{j=1} v(i) \cdot v(j) \cdot P\left( \{X_0 \in C(i)\} \cap \{ X_m \in C(j)\}\right).
\end{align}

Also, for each $i\in \{1,2,\dots,\ell\}$, again noting that $I_{C(i)}(X_0) = I(X_0 \in C(i))$ (where again the second indicator function is defined on $\Omega$), one has that $E(I_{C(i)}(X_0)) = P(X_0 \in C(i)) = \mu(C(i))$, (where of course the last equality comes from Context \ref{ctxt5.1}(B)). Hence by (\ref{eq5.5.9}), and (\ref{eq5.5.6}),
\[
0 = E[g(X_0)] = E\left[ \sum^\ell_{i=1} v(i) \cdot I_{C(i)} (X_0)\right] = \sum^\ell_{i=1} v(i) \cdot \mu(C(i)).
\]
Hence
\[
0 = 0 \cdot 0 = \left(\sum^\ell_{i=1} v(i) \cdot \mu (C(i))\right) \cdot \left( \sum^\ell_{j=1}v(j) \cdot \mu(C(j))\right) = \sum^\ell_{i=1} \sum^\ell_{j=1} v(i)\cdot v(j)\cdot \mu(C(i))\cdot \mu(C(j)).
\]
Subtracting that $0$ from the far right side of (\ref{eq5.5.13}), one has by (\ref{eq5.5.13}) itself that for each positive integer $m$,
\begin{align}\label{eq5.5.14}
E \left[g(X_0) \cdot g(X_m)\right]&= \sum^\ell_{i=1} \sum^\ell_{j=1} v(i) \cdot v(j) \cdot P\left( \{X_0\in C(i)\}\cap \{ X_m\in C(j)\}\right) \nonumber\\
&\qquad -\sum^\ell_{i=1} \sum^\ell_{j=1} v(i) \cdot v(j) \cdot \mu (C(i)) \cdot \mu (C(j)) \nonumber\\
&= \sum^\ell_{i=1} \sum^\ell_{j=1} v(i) \cdot v(j) \cdot \left[ P(\{X_0\in C(i)\} \cap \{X_m \in C(j)\}) - \mu(C(i)) \cdot \mu(C(j))\right].
\end{align}

Now by (\ref{eq5.5.1}), if $i$ and $j$ (equal or distinct) are each an element of $\{1,2,\dots,\ell\}$, then
\[
\left| P(\{X_0 \in C(i)\}\cap \{X_{\doub(n)} \in C(j)\}) - \mu(C(i)) \cdot \mu(C(j))\right| \ll r^{\doub(n)} \,\,\,\mbox{as $n\to \infty$.}
\]
Hence
\[
\sum^\ell_{i=1}\sum^\ell_{j=1} |v(i) \cdot v(j)| \cdot \left| P(\{X_0 \in C(i)\} \cap \{X_{\doub(n)} \in C(j)\}) - \mu (C(i)) \cdot \mu (C(j))\right| \ll r^{\doub(n)}\,\,\mbox{as\,\,$n\to \infty$.}
\]
Hence by (\ref{eq5.5.14}) (and (\ref{eq5.3.5})),
\[
E\left[ g(X_0) \cdot g(X_{\doub(n)})\right] \ll r^{\doub(n)} \,\,\mbox{as $n\to \infty$}.
\]
That is,
\begin{equation}\label{eq5.5.15}
\limsup_{n\to \infty} E\left[ g(X_0) \cdot g(X_{\doub (n)})\right]/r^{\doub(n)} <\infty.
\end{equation}

However, eqs.\ (\ref{eq5.5.15}) and (\ref{eq5.5.12}) contradict each other. Hence eq.\ (\ref{eq5.5.2}) must be false. Instead $\rho(X,1) \le r$ after all. That completes the proof of Lemma \ref{lm5.5}.
\end{proof}

\begin{remark}\label{rm5.6}
By Lemma \ref{lm5.5}, one has that Proposition \ref{pr3.5}(I) holds.

The following ``corollary'' of Lemma \ref{lm5.5} will give a little more information.
\end{remark}

\begin{corollary}\label{co5.7}
In Context \ref{ctxt5.1}, the following holds:

For any given $r\in (0,1)$, the following four conditions (i), (ii), (iii), (iv) are equivalent:

\noindent (i) $\rho(X,1) \le r$.

\noindent (ii) $\rho(X,n) \le r^n$ for all $n\in \N$.

\noindent (iii) $\alpha(X,n) \le r^n$ for all $n\in \N$.

\noindent (iv) For every pair of sets $A,B\in \cS$, eq.\ (\ref{eq5.5.1}) holds.
\end{corollary}

Corollary \ref{co5.7} is in some sense an ``analog'' of --- and is at least in spirit contained in --- the result of Roberts and Tweedie [\cite{ref-journal-RT}, Theorem 3], building on Roberts and Rosenthal [\cite{ref-journal-RR}, Proof of Theorem 2.1], that was alluded to in Remark \ref{rm3.6}(C) and Remark \ref{rm4.10} --- the ``matching of rates'' result in connection with ``Theorem \ref{th1.1}.1''/Remark \ref{rm4.7}.

\begin{proof}[Proof of Corollary \ref{co5.7}]
Suppose $r\in (0,1)$. Then (i) $\Longrightarrow$ (ii) by (\ref{eq2.7.6}), (ii) $\Longrightarrow$ (iii) by (\ref{eq2.4.1}), (iii) $\Longrightarrow$ (iv) trivially (with [LHS of (\ref{eq5.5.1})]\ $\le \alpha(X,2^n) \le r^{\doub(n)}$ for each $n\in \N$), and (iv) $\Longrightarrow$ (i) by Lemma \ref{lm5.5}. Thus Corollary \ref{co5.7} holds.
\end{proof}

\section{Review of a proof of Proposition \ref{pr3.5}(III)}
\label{sc6}
This section is devoted primarily to a review of a proof of the observation taken essentially from Longla [\cite{ref-journal-Longla14}, Lemma 2.1] that was stated in Proposition \ref{pr3.5}(III). It is stated again here for convenient reference:

\begin{proposition}[cf.\ Longla \textrm{\cite{ref-journal-Longla14}, Lemma 2.1}] \label{pr6.1}  In Context \ref{ctxt5.1} (that is, for a given strictly stationary, reversible Markov chain $X:= (X_k, k\in \Z)$ with state space $S$), one has that for every positive integer $n$,
\begin{equation}\label{eq6.1.1}
\rho(X,n) = [\rho(X,1)]^n.
\end{equation}
\end{proposition}

As stated after Proposition \ref{pr3.5}, the argument below has the spirit of Longla's argument, but its presentation here will be ``grounded'' in the material in Section \ref{sc5}.

\begin{proof}
If $\rho(X,1) =0$, then $X$ is a sequence of independent, identically distributed, $S$-valued random variables, hence $\rho(X,n) =0$ for every $n\in\N$, and hence (\ref{eq6.1.1}) holds for all $n\in\N$ and we are done.

Therefore, henceforth assume
\begin{equation}\label{eq6.1.2}
\rho(X,1)>0.
\end{equation}
(Of course that includes the possibility that $\rho(X,1) =1$.) The rest of the proof will proceed through two lemmas and then a final argument.

\medskip 
\noindent \textbf{Lemma 1.}\label{lm1}
\textit{Refer to (\ref{eq6.1.2}). Suppose $n$ is a positive integer, and $t$ is a number such that
\begin{equation}\label{eq6.1.3}
0<t<\rho(X,1).
\end{equation}
Then}
\begin{equation}\label{eq6.1.4}
\rho(X,2^n) >t^{\doub(n)}.
\end{equation}
\medskip

{\textbf{ Proof of Lemma 1}.} Referring to (\ref{eq6.1.3}) and applying (\ref{eq2.4.4}), let $V$ and $W$ be real-valued simple random variables such that $V$ is measurable with respect to $\sigma(X_0)$, $W$ is measurable with respect to $\sigma(X_1)$,  and
\begin{equation}\label{eq6.1.5}
E(V^2) \le 1,\,\,E(W^2)\le 1,\,\,EV = EW =0,\,\,\mbox{and}\,\,E(VW) >t.
\end{equation}

Employing the procedure in Remark \ref{rm5.4}, let $g:S\to \R$ and $h:S\to \R$ be $\cS/\cR$-measurable simple functions such that $V = g(X_0)$ and $W = h(X_1)$. Of course
\begin{equation}\label{eq6.1.6}
\int_\R g^2d\mu = E\left[(g(X_0))^2\right] = E(V^2) \le 1,
\end{equation}
and
\begin{equation}\label{eq6.1.7}
\int_\R gd\mu= E\left[ g(X_0)\right] = EV =0,
\end{equation}
and hence $g \in \cL_{ub0}(\mu)$. Similarly $h\in \cL_{ub0}(\mu)$.  Also,
\begin{equation}\label{eq6.1.8}
E\left[ g(X_0) \cdot h(X_1)\right] = E(VW) >t.
\end{equation}

Now refer to the positive integer $n$ in the statement of Lemma 1. By (\ref{eq6.1.8}) and (\ref{eq5.3.9}),
\[
t< E\left[ g(X_0) \cdot h(X_1)\right] \le \left( E[g(X_0) \cdot g(X_{\doub(n)})]\right)^{1/\doub(n)},
\]
and hence
\begin{equation}\label{eq6.1.9}
t^{\doub(n)} < E\left[ g(X_0) \cdot g(X_{\doub(n)})\right].
\end{equation}
By strict stationarity and (\ref{eq6.1.6}) and (\ref{eq6.1.7}), one has that $E[(g(X_{\doub(n)}))^2] = E[(g(X_0))^2]\le 1$ and \newline \noindent $E[g(X_{\doub(n)})] = E[g(X_0)]=0$.  Hence by (\ref{eq6.1.6}), (\ref{eq6.1.7}), (\ref{eq6.1.9}), and simple arithmetic,
\begin{align}
\corr\left(g(X_0),g(X_{\doub(n)})\right) &\ge \cov \left(g(X_0),g(X_{\doub(n)})\right) \nonumber\\
&=E\left[g(X_0)\cdot g(X_{\doub(n)})\right] >t^{\doub(n)}.
\end{align}
Hence (\ref{eq6.1.4}) holds.  That completes the proof of Lemma 1.

\medskip
\noindent \textbf{Lemma 2.}\label{lm2} \textit{Refer to (\ref{eq6.1.2}). Suppose $n$ is a positive integer. Then}
\begin{equation}\label{eq6.1.10}
\rho\left(X,2^n\right) = [\rho(X,1)]^{\doub(n)}.
\end{equation}
\medskip

{\textbf{Proof of Lemma 2.}} By Lemma 1,
\begin{equation}\label{eq6.1.11}
\rho(X,2^n) \ge \lim_{t\to \rho(X,1)-} t^{\doub(n)} = [\rho(X,1)]^{\doub(n)}.
\end{equation}
On the other hand, by (\ref{eq2.7.6}), $\rho(X,2^n) \le [\rho(X,1)]^{\doub(n)}$. Combining that with (\ref{eq6.1.11}), one obtains (\ref{eq6.1.10}). That completes the proof of Lemma 2.
\end{proof}

\textbf{Conclusion of proof of Proposition \ref{pr6.1}.} Eq.\ (\ref{eq6.1.1}) holds trivially for $n=1$, and it holds for $n\in \{2,4,8,16,32,\dots\}$ by Lemma 2. Let
\[
\ell \in \N -\{1,2,4,8,16,32,\dots\}
\]
be arbitrary but fixed.  To complete the proof of Proposition \ref{pr6.1}, it suffices to show that (\ref{eq6.1.1}) holds for $n=\ell$, that is, to show that
\begin{equation}
\label{eq6.1.12}
\rho(X,\ell) = [\rho(X,1)]^\ell.
\end{equation}

Let $j$ be a positive integer sufficiently large that
\begin{equation}\label{eq6.1.13}
\ell <2^j.
\end{equation}

By Lemma 2 and eqs.\ (\ref{eq6.1.13}), (\ref{eq2.7.4}), and (\ref{eq2.7.6}),
\[
[\rho(X,1)]^{\doub(j)} = \rho(X,2^j)  \le \rho(X,\ell) \cdot \rho(X,2^j - \ell) \le \rho(X,\ell) \cdot [\rho(X,1)]^{[\doub(j)]-\ell}.
\]
Hence
\begin{equation}\label{eq6.1.14}
[\rho(X,1)]^\ell \le \rho(X,\ell).
\end{equation}
On the other hand, by (\ref{eq2.7.6}), $\rho(X,\ell) \le [\rho(X,1)]^\ell$. Combining that with (\ref{eq6.1.14}), one obtains (\ref{eq6.1.12}).  That completes the proof of Proposition \ref{pr6.1} (that is, Proposition \ref{pr3.5}(III)).
$\square$ \hfill\break

\section{Review of a proof of Proposition \ref{pr3.5}(II)}
\label{sc7}

The setting for the work here in Section \ref{sc7} will be Context \ref{ctxt5.1} together with assumption (A3) in Theorem \ref{th3.3}. That combination will be recorded here for convenient reference:

\begin{context} \label{ctxt7.1} (A) Suppose $X:= (X_k, k\in \Z)$ is a strictly stationary Markov chain with state space $S$ (see Convention \ref{cn2.2}(B)(D) again), such that $X$ is reversible (but not necessarily irreducible).

(B) Let $\mu$ denote the (marginal) distribution (on $(S,\cS))$ of the $S$-valued random variable $X_0$.

(C) Retain all notations from Context \ref{ctxt5.1}(C), and note all observations there.

(D) Note all observations in Context \ref{ctxt5.1}(D).

(E) Assume condition (A3) in Theorem \ref{th3.3}. That is, suppose that for every set $A\in \cS$, there exists a number $c_A\in (0,1)$ such that (see (\ref{eq2.1.1}))

\begin{equation}
\label{eq7.1.1}
\left|P(\{X_0\in A\} \cap \{X_{\doub(n)} \in A\})-[\mu(A)]^2\right| \ll c_A^{\doub(n)} \,\,\hbox{as $n\to \infty$}.
\end{equation}
\end{context}

\begin{remark}\label{rm7.2}
Our goal here in Section \ref{sc7} is to prove that in Context \ref{ctxt7.1}, under all of its assumptions, condition (A2) in Theorem \ref{th3.3} holds:

There exists a number $r\in (0,1)$ such that for every pair of sets $A\in \cS$ and $B\in \cS$, one has (see (\ref{eq2.1.1})) that
\begin{equation}
\label{eq7.2.1}
\left|P(\{X_0 \in A\}\cap \{X_{\doub(n)} \in B\}) - \mu(A) \mu(B)\right| \ll r^{\doub(n)}\,\,\mbox{as\,\,$n\to \infty$.}
\end{equation}
Once that is accomplished, the proof of Proposition \ref{pr3.5}(II) will be complete.
\end{remark}

\begin{definition}\label{df7.3}  For every function $g\in \cL^2_{ub}(\mu)$ (see (\ref{eq5.1.1})), referring to (\ref{eq5.3.5}) and (\ref{eq5.3.7}), define the number $r(g) \in [0,1]$ as follows:
\begin{equation}
\label{eq7.3.1}
r(g) := \sup_{n\in \N} \left[\left( E\left[g(X_0) \cdot g(X_{\doub(n)})\right]\right)^{1/\doub(n)}\right]
= \lim_{n\to \infty}\left[\left(E\left[g(X_0) \cdot g(X_{\doub(n)})\right]\right)^{1/\doub(n)}\right]
\end{equation}
\end{definition}

\begin{remark}\label{rm7.4}  Let us record here for convenient later reference the following facts.

By (\ref{eq5.3.5}), (\ref{eq5.3.7}), and (\ref{eq7.3.1}),

\begin{equation}\label{eq7.4.1}
\mbox{For every $g\in \cL^2_{ub}(\mu)$,} \quad 0\le r(g) \le 1.
\end{equation}
Also, by (\ref{eq5.3.5}) and (\ref{eq7.3.1}) (its first equality),
\begin{equation}\label{eq7.4.2}
\mbox{for every $n\in \N$ and every $g\in \cL^2_{ub}(\mu)$,}\,\,\, 0\le E\left[g(X_0) \cdot g(X_{\doub(n)})\right] \le [r(g)]^{\doub(n)}.
\end{equation}
\end{remark}

\begin{lemma}\label{lm7.5} Suppose $g\in \cL^2_{ub}(\mu)$. Suppose $a$ is an element of $[0,1]$. Refer to (\ref{eq7.4.1}) and (\ref{eq7.4.2}). Then the following four conditions are equivalent:

\noindent (i) $E\left[ g(X_0) \cdot g(X_{\doub(n)})\right] \ll a^{\doub(n)}$ as $n\to \infty$.

\noindent (ii) $E\left[g(X_0) \cdot g(X_{\doub(n)})\right] \le a^{\doub(n)}$ for every $n\in\N$.

\noindent(iii) $\left(E\left[g(X_0)\cdot g(X_{\doub(n)})\right]\right)^{1/\doub(n)} \le a$ for every $n\in\N$.

\noindent(iv) $r(g) \le a$.
\end{lemma}

\begin{proof} Refer again to (\ref{eq7.4.1}) and (\ref{eq7.4.2}). Trivially conditions (ii) and (iii) (in Lemma \ref{lm7.5}) are equivalent. By (\ref{eq7.3.1}) (its first equality), conditions (iii) and (iv) are equivalent. Trivially condition (ii) implies condition (i). To complete the proof of Lemma \ref{lm7.5}, it suffices to show that condition (i) implies condition (iv).

\textbf{Proof that (i) $\Longrightarrow$ (iv).}  Suppose condition (i) holds. Suppose condition (iv) \textit{fails} to hold; that is, suppose instead that $a< r(g)$. We shall aim for a contradiction.

Recall the hypothesis $a\in [0,1]$. Let $t$ be a (positive) number such that
\begin{equation}\label{eq7.5.1}
a< t< r(g).
\end{equation}
Then $0\le a<t$ and hence $a/t<1$ and $a^m/t^m = (a/t)^m \to 0$ as $m\to \infty$. In particular, $a^{\doub(n)}/t^{\doub(n)} \to 0$ as $n\to \infty$. Hence by the assumption of condition (i),
\begin{equation}\label{eq7.5.2}
E\left[ g(X_0) \cdot g(X_{\doub(n)})\right] = o\left(t^{\doub(n)}\right)\quad \mbox{as $n\to \infty$.}
\end{equation}

Now by (\ref{eq7.5.1}) and (\ref{eq7.3.1}), there exists a positive integer $N$ such that for every integer $n\ge N$, $E[g(X_0) \cdot g(X_{\doub(n)})]^{1/\doub(n)} >t$ and hence $E[g(X_0) \cdot g(X_{\doub(n)})] >t^{\doub(n)}$. But that contradicts (\ref{eq7.5.2}). Hence it must be the case after all that (if condition (i) holds) condition (iv) must hold. That completes the proof that (i) $\Longrightarrow$ (iv), and the proof of Lemma \ref{lm7.5}.
\end{proof}

\begin{definition}\label{df7.6}  For every set $A\in \cS$, define the (``centered indicator'') function $J_A:S \to \R$ as follows:

For every $s\in S$,
\begin{equation}\label{eq7.6.1}
J_A(s) = I_A(s) - \mu(A)
\end{equation}
(where $I_A :S \to \{0,1\}$ denotes the indicator function of the set $A$).
\end{definition}

\begin{remark}\label{rm7.7}  The comments in this Remark are downright trivial, but are being put on record here for later convenient reference.

(A) For a given set $A\in \cS$, the following comments hold: For any $s\in S$,
\[
-1 \le -\mu(A) \le I_A(s) - \mu(A) \le 1-\mu(A) \le 1,
\]
hence $|J_A(s) | = |I_A(s) - \mu(A)| \le 1$, hence
\[
\int_S J^2_Ad\mu \le \int_S1^2d\mu = \mu(S) = 1.
\]
Also,
\[
\int_SJ_Ad\mu = \int_S(I_A - \mu(A)) d\mu =\int_S I_Ad\mu - \int_S \mu(A)d\mu = \mu(A) - \mu(A) \cdot \mu (S)
= \mu(A) - \mu(A) \cdot 1 =0.
\]
From these two observations, one has that
\begin{equation}\label{eq7.7.1}
\mbox{for all $A\in \cS$,} \quad J_A \in \cL^2_{ub0}(\mu).
\end{equation}

(B) For a given $A\in \cS$ and a given $k\in Z$, one of course has the equality of $\{0,1\}$-valued random variables
\begin{equation}\label{eq7.7.2}
I_A(X_k) = I(X_k\in A)
\end{equation}
(again, as in the second paragraph after eq.\ (\ref{eq5.2.2})).
For a given $A\in \cS$ and a given $k\in \Z$, by (\ref{eq7.7.2}),
\begin{equation}
\label{eq7.7.3}
E[I_A(X_k)] = E[I(X_k\in A)] = P(X_k\in A) = \mu(A).
\end{equation}
(Alternatively, $E[I_A(X_k)] = \int_S I_Ad\mu = \mu(A)$.)

(C) For any given $k\in \Z$ and any given pair of sets $A \in \cS$ and $B\in \cS$,
\begin{align}\label{eq7.7.4}
E\left[J_A(X_0) \cdot J_B(X_k)\right]&=E\left[\left(I_A(X_0) - \mu(X)\right)\cdot \left(I_B(X_k) - \mu(B)\right)\right] \nonumber\\
&=E\left[I_A(X_0) \cdot I_B(X_k)\right] - E\left[I_A(X_0) \cdot \mu(B)\right] -E\left[\mu(A) \cdot I_B(X_k)\right] + \mu(A)\cdot \mu(B) \nonumber\\
&=E\left[I(X_0\in A) \cdot I(X_k\in B)\right] - \mu(B) \cdot E\left[I_A(X_0)\right] - \mu(A) \cdot E\left[I_B(X_k)\right] +\mu(A)\cdot \mu(B) \nonumber\\
&=E\left[I\left(\left\{X_0\in A\right\} \cap \left\{X_k\in B\right\}\right)\right] - \mu(B) \cdot\mu(A) -\mu(A) \cdot \mu(B) + \mu(A) \cdot \mu(B) \nonumber\\
&= P\left(\left\{X_0\in A\right\} \cap \left\{ X_k \in B\right\}\right) - \mu(A) \cdot \mu(B).
\end{align}

(D) For any given $\ell\in \Z$ and any given pair of sets $A\in \cS$
 and $B\in\cS$, one has that $\mu(A) \cdot \mu(B)\le \min        \{\mu(A), \mu(B)\}$, and hence by (\ref{eq7.7.4}),
 \begin{align}
-\min&\{\mu(A),\mu(B)\} \le -\mu(A) \cdot \mu(B)\nonumber\\
 &\le P\left(\{X_0\in A\} \cap \{X_\ell \in B\}\right) -\mu(A) \cdot \mu(B) =E\left[J_A(X_0) \cdot J_B(X_\ell)\right]\nonumber\\
 &\le P\left( \{X_0\in A\} \cap \{X_\ell \in B\}\right)\le \min \left\{P(X_0\in A), P(X_\ell \in B)\right\} =\min\{\mu(A) \cdot \mu(B) \},\nonumber
 \end{align}
 that is,
\[
 -\min \{\mu(A),\mu(B)\} \le E\left[J_A(X_0) \cdot J_B(X_\ell)\right]\le \min \{\mu(A),\mu(B)\}.
 \]
Thus for all $\ell \in \Z$, for all $A\in \cS$ and $B\in \cS$, one has that
 \begin{equation}\label{eq7.7.5}
 \left|E\left[J_A(X_0) \cdot J_B(X_\ell)\right]\right| \le \min \{\mu(A),\mu(B)\}.
 \end{equation}
\end{remark}

\begin{remark}\label{rm7.8}
(A) Refer to (\ref{eq7.1.1}) and (\ref{eq7.7.4}). By (\ref{eq7.7.4}), the assumption of condition (A3) in Theorem \ref{th3.3}, as formulated in Context \ref{ctxt7.1}(E) (with eq.\ (\ref{eq7.1.1})), can be formulated as follows: For every set $A\in \cS$, there exists a number $c_A\in (0,1)$ such that (see (\ref{eq2.1.1}))
\begin{equation}\label{eq7.8.1}
E\left[J_A(X_0) \cdot J_A(X_{\doub(n)})\right] \ll c^{\doub(n)}_A\,\,\,\mbox{as $n\to \infty$.}
\end{equation}

(B) By Lemma \ref{lm7.5}, the assumption of condition (A3), as formulated in the entire sentence containing (\ref{eq7.8.1}), can be reformulated again, as follows:

For every set $A\in \cS$, there exists a number $c_A\in (0,1)$ such that (see (\ref{eq2.1.1})), $r(J_A) \le c_A$. Thus obviously condition (A3) can be reformulated as follows:
\begin{equation}\label{eq7.8.2}
\mbox{For all $A\in \cS$,}\quad r(J_A) <1.
\end{equation}

(C) Eq.\ (\ref{eq7.8.2}) is our new, simplified formulation of condition (A3) (in Theorem \ref{th3.3}) -- that is, Context \ref{ctxt7.1}(E). Thus (see Remark \ref{rm7.2}) our task here is Section \ref{sc7} is to prove that in Context \ref{ctxt7.1}(A)(B)(C)(D) -- that is, in Context \ref{ctxt5.1} -- if (\ref{eq7.8.2}) holds then condition (A2) (in Theorem \ref{th3.3}) holds.
\end{remark}

\begin{definition}\label{df7.9}  Refer to Definition \ref{df7.6}, eq.\ (\ref{eq7.7.1}), Definition \ref{df7.3}, and eq.\ (\ref{eq7.4.1}). For every set $D\in \cS$, define the number $R(D) \in [0,1]$ by
\begin{equation}\label{eq7.9.1}
R(D) := \sup r(J_A)
\end{equation}
where the supremum is taken over all sets $A\in \cS$ such that $A\subset D$.
\end{definition}

\begin{remark}\label{rm7.10}
(A) By (\ref{eq7.9.1}) and (\ref{eq7.4.1}), one has the following:
\begin{equation}\label{eq7.10.1}
\mbox{For every set $D\in \cS$,} \quad 0 \le r(J_D) \le R(D) \le 1.
\end{equation}

(B) As a trivial consequence of (\ref{eq7.9.1}) (simply involving taking the supremum over a larger class of sets), one  has the following:
\begin{equation}\label{eq7.10.2}
\mbox{If $B,D \in \cS$ and $B\subset D$, then}\,\,\,R(B) \le R(D).
\end{equation}
\end{remark}

\textbf{Note.} The next six lemmas will be devoted to proving that (see Definition \ref{df7.9}) $R(S) <1$. Once that is accomplished, the rest of the task here in Section \ref{sc7} -- the presentation of a proof of Proposition \ref{pr3.5}(II), as outlined in the scheme described in Remark \ref{rm7.8}(C) -- will take only a very little more work.
\medskip

\begin{lemma}\label{lm7.11} Refer to Definition \ref{df7.6}, eq.\ (\ref{eq7.7.1}), Definition \ref{df7.3}, and eq.\ (\ref{eq7.4.1}). In Context \ref{ctxt7.1}, if $A$ and $B$ are disjoint sets $\in \cS$, then the following three statements hold:

(I) The functions $J_A : S\to \R$, $J_B : S\to \R$, and $J_{A \cup B}: S \to \R$ (see (\ref{eq7.6.1})) satisfy
\begin{equation}\label{eq7.11.1}
J_{A\cup B}(s) = J_A(s) +J_B(s)\quad\mbox{for every $s\in S$.}
\end{equation}

(II) One has that
\begin{equation}\label{eq7.11.2}
r(J_{A\cup B}) \le \max\{r(J_A),r(J_B)\}.
\end{equation}

(III) If also $r(J_A) \not= r(J_B)$, then
\begin{equation}\label{eq7.11.3}
r(J_{A\cup B}) = \max\{ r(J_A), r(J_B)\}.
\end{equation}
\end{lemma}

\textbf{Proof of (I).} Suppose $s\in S$. Since by hypothesis the sets $A$ and $B$ are disjoint, the indicator functions of $A$, $B$, and $A\cup B$ satisfy $I_{A\cup B}(s) = I_A(s) + I_B(s)$. Hence by (\ref{eq7.6.1}) and again the hypothesis that $A$ and $B$ are disjoint,
\begin{align}
J_{A\cup B}(s) &= I_{A\cup B}(s)  - \mu(A\cup B) = [I_A(s) +I_B(s)] -[\mu(A) + \mu(B)] \nonumber\\
&= [I_A(s) -\mu(A) ] + [I_B(s) - \mu(B)] =J_A(s) +J_B(s).\nonumber
\end{align}
Thus (\ref{eq7.11.1}) holds -- that is Statement I holds.
\medskip

\textbf{Preparation for the proofs of Statements (II) and (III).}

Here we shall carry out some calculations that will be pertinent to the proofs of both (II) and (III).

By (\ref{eq7.11.1}), for (say) every $m\in \N$,
\begin{align} \label{eq7.11.4}
E&\left[J_{A\cup B}(X_0) \cdot J_{A\cup B}(X_k)\right] = E\left[(J_A(X_0)) + (J_B(X_0)) \cdot (J_A(X_m) + J_B(X_m))\right]\nonumber\\
&=E\left[J_A(X_0) \cdot J_A(X_m)\right] + E\left[J_B(X_0) \cdot J_B(X_m)\right] + E\left[ J_A(X_0) \cdot J_B(X_m)\right] +E\left[ J_A(X_m) \cdot J_B(X_0)\right].
\end{align}

For any given $m\in \N$, the following statements hold:

By reversibility (see Context \ref{ctxt7.1} again), the random vectors $(X_0, X_m)$ and $(X_m,X_0)$ have the same distribution (on $(S^2, \cS^2))$.  Hence the random variables $J_A(X_0) \cdot J_B(X_m)$ and $J_A(X_m) \cdot J_B(X_0)$ have the same distributions (on $(\R,\cR))$. Hence $E[J_A(X_0) \cdot J_B(X_m)] = E[J_A(X_m) \cdot J_B(X_0)]$.

Hence by (\ref{eq7.11.4}), for any $m\in \N$,
\begin{align}\label{eq7.11.5}
E\left[J_{A \cup B}(X_0) \cdot J_{A\cup B}(X_m)\right] = &E\left[J_A(X_0) \cdot J_A(X_m)\right] + E\left[ J_B(X_0) \cdot J_B(X_m)\right] \nonumber\\
&+2 \cdot E\left[ J_A(X_0) \cdot J_B(X_m)\right].
\end{align}

By (\ref{eq7.7.1}) (applied to $A$ and to $B$) and Remark \ref{rm5.3}(B), for any $m\in \N$, one has that
\[
\left| E\left[J_A(X_0) \cdot J_B(X_m)\right]\right| \le \left(E\left[ J_A(X_0) \cdot J_A(X_{2m})\right]\right)^{1/2}.
\]
Hence by (\ref{eq7.3.1}) (its first equality), for any $n\in \N$,
\begin{align}\label{eq7.11.6}
&\left|E\left[J_A(X_0) \cdot J_B\left(X_{\doub(n)}\right)\right]\right| \le \left(E\left[J_A(X_0) \cdot J_A\left(X_{\doub(n+1)}\right)\right]\right)^{1/2} \nonumber\\
&\le \left([r(J_A)]^{\doub(n+1)}\right)^{1/2} = \left[ r(J_A)\right]^{\doub(n)}.
\end{align}
By an analogous argument and (\ref{eq5.3.3}), for any $n\in \N$,
\begin{equation}\label{eq7.11.7}
0\le E\left[J_A(X_0) \cdot J_A(X_{\doub(n)})\right] \le [r(J_A)]^{\doub(n)}
\end{equation}
and
\begin{equation}
\label{eq7.11.8}
0 \le E\left[J_B(X_0) \cdot J_B(X_{\doub(n)})\right] \le \left[r(J_B)\right]^{\doub(n)}.
\end{equation}
Applying (\ref{eq7.11.6}), (\ref{eq7.11.7}), and (\ref{eq7.11.8}) to (\ref{eq7.11.5}), one has that for any $n\in\N$,
\begin{equation}\label{eq7.11.9}
E\left[J_{A\cup B}(X_0) \cdot J_{A\cup B}(X_{\doub(n)})\right] \le 3 \cdot [r(J_A)]^{\doub(n)} + [r(J_B)]^{\doub(n)}
\le 4\cdot \left[\max\{r(J_A),r(J_B)\}\right]^{\doub(n)}.
\end{equation}

\textbf{Proof of Statement (II).} By (the entire sentence of) (\ref{eq7.11.9}), one has that
\[
E\left[J_{A\cup B}(X_0) \cdot J_{A\cup B} (X_{\doub(n)})\right] \ll \left[ \max \{r(J_A), r(J_B)\}\right]^{\doub(n)} \,\,\,\mbox{as $n\to \infty.$}
\]
Hence by Lemma \ref{lm7.5} (and eq.\ (\ref{eq7.4.1}) and (\ref{eq7.7.1}), giving $\max\{r(J_A),r(J_B)\} \le 1$), and also eq.\ (\ref{eq7.7.1}) applied to the set $A\cup B$), one  has that eq.\ (\ref{eq7.11.2}) holds. That is, Statement (II) holds.
\medskip

 \textbf{Proof of Statement (III).} Now as in Statement (III), in addition to the assumption that $A$ and $B$ are disjoint members of $\cS$, assume that $r(J_A) \not= r(J_B)$.  We assume without loss of generality that
\[
r(J_A) <r(J_B).
\]

Let $t$ be a real number such that $r(J_A) <t < r(J_B)$. By (\ref{eq7.4.1}),
\begin{equation}\label{eq7.11.10}
0 \le r(J_A) < t< r(J_B) \le 1.
\end{equation}

Of course, since (by (\ref{eq7.11.10})) $0\le r(J_A)/t <1$, one has that
\begin{equation}
\label{eq7.11.11}
[r(J_A)]^{\doub(n)}/t^{\doub(n)} = [r(J_A)/t]^{\doub(n)} \longrightarrow 0 \quad\mbox{as $n\to \infty$.}
\end{equation}

By (\ref{eq7.11.10}) and (\ref{eq7.3.1}) (and (\ref{eq5.3.3})), one has that for all $n\in \N$ sufficiently large,
\[
\left(E\left[J_B(X_0) \cdot J_B(X_{\doub(n)})\right]\right)^{1/\doub(n)} >t >0
\]
and hence
\begin{equation}\label{eq7.11.12}
E\left[ J_B(X_0) \cdot J_B(X_{\doub(n)}) \right] > t^{\doub(n)} >0.
\end{equation}
By (\ref{eq7.11.7}) and (\ref{eq7.11.11}),
\[
\left(E\left[ J_A(X_0) \cdot J_A(X_{\doub(n)})\right]\right)/t^{\doub(n)} \longrightarrow 0 \,\,\,\mbox{as $n\to \infty$.}
\]
Hence by the entire sentence containing (\ref{eq7.11.12}) (note the phrase ``for all $n\in\N$ sufficiently large'' in that sentence),
\begin{equation}
\label{eq7.11.13}
\frac{ E\left[J_A(X_0) \cdot J_A(X_{\doub(n)})\right]}{E\left[ J_B(X_0) \cdot J_B(X_{\doub(n)})\right]} \,\,\longrightarrow 0 \,\,\,\mbox{as $n\to \infty$.}
\end{equation}
Similarly, by (\ref{eq7.11.6}) and (\ref{eq7.11.11}),
\[
\left(E\left[ J_A(X_0) \cdot J_B(X_{\doub(n)})\right]\right)/t^{\doub(n)} \longrightarrow 0 \,\,\,\mbox{as $n\to \infty$,}
\]
and hence again by the entire sentence containing (\ref{eq7.11.12}),
\begin{equation}\label{eq7.11.14}
\frac{E\left[J_A(X_0) \cdot J_B(X_{\doub(n)})\right]}{E\left[J_B(X_0) \cdot J_B(X_{\doub(n)})\right]} \longrightarrow 0 \,\,\,\mbox{as $n\to \infty$.}
\end{equation}
By (\ref{eq7.11.5}), (\ref{eq7.11.13}), and (\ref{eq7.11.14}),
\begin{equation}\label{eq7.11.15}
\frac{E\left[ J_{A\cup B}(X_0) \cdot J_{A\cup B}(X_{\doub(n)})\right]}{E\left[J_B(X_0) \cdot J_B(X_{\doub(n)})\right]} \longrightarrow 1 \,\,\,\mbox{as $n\to \infty$.}
\end{equation}

For each $n\in \N$, just for convenience, let $\theta_n$ denote the left side of (\ref{eq7.11.15}). For any given $\varepsilon \in (0,1)$, one has the following:  First, by (\ref{eq7.11.15}) itself, $1-\varepsilon < \theta_n < 1+\varepsilon$ for all $n\in \N$ sufficiently large; and then by trivial arithmetic for all such (sufficiently large) $n\in \N$,
\[
1-\varepsilon <(1-\varepsilon)^{1/\doub(n)}<\theta_n^{1/\doub(n)} <(1+\varepsilon)^{1/\doub(n)} <1+\varepsilon.
\]
It follows that $\theta^{1/\doub(n)}_n \longrightarrow 1$ as $n\to \infty$.

Let us display that last fact (again see (\ref{eq7.11.15})):
\begin{equation}\label{eq7.11.16}
\frac{\left(E\left[ J_{A\cup B}(X_0) \cdot J_{A\cup B}(X_{\doub(n)})\right]\right)^{1/\doub(n)}}{\left(E\left[J_B(X_0) \cdot J_B(X_{\doub(n)})\right]\right)^{1/\doub(n)}} \longrightarrow 1 \,\,\,\mbox{as $n\to \infty$.}
\end{equation}
By (\ref{eq7.3.1}), one has that in the fraction in the left side of (\ref{eq7.11.16}), the denominator converges to $r(J_B)$ as $n\to \infty$.  (Of course $0<r(J_B) \le 1$ by (\ref{eq7.11.10}).) Hence by (\ref{eq7.11.16}) itself, the numerator in that fraction converges to $r(J_B)$ as $n\to \infty$. Of course by (\ref{eq7.3.1}), the numerator in that fraction also converges to $r(J_{A\cup B})$ as $n\to \infty$. Hence $r(J_{A\cup B}) = r(J_B)$. Hence by (\ref{eq7.11.10}), eq.\ (\ref{eq7.11.3}) holds.  Thus Statement (III) holds. That completes the proof of Lemma \ref{lm7.11}.
\hfill$\square$\break

\begin{lemma}\label{lm7.12} Suppose $C$ and $D$ are disjoint sets $\in \cS$. Then (see (\ref{eq7.9.1}) and (\ref{eq7.10.1}))
\begin{equation}\label{eq7.12.1}
R(C\cup D) = \max\{R(C),R(D)\}.
\end{equation}
\end{lemma}

\begin{proof} By (\ref{eq7.10.2}), $R(C) \le R(C\cup D)$ and $R(D) \le R(C\cup D)$. Hence
\begin{equation}\label{eq7.12.2}
R(C\cup D) \ge \max\{R(C),R(D)\}.
\end{equation}
Our task now is to prove the opposite inequality.

Since the sets $C$ and $D$ are disjoint (by hypothesis), one has that for any set $A\in \cS$ such that $A\subset C\cup D$, the following observations hold: The sets $A\cap C$ and $A\cap D$ are disjoint and their union is $A$; and hence by Lemma \ref{lm7.11}(II) and then (\ref{eq7.9.1}),
\[
r(J_A) = r\left( J_{(A\cap C)\cup(A\cap D)}\right) \le \max \left\{ r(J_{A\cap C}),r(J_{A\cap D})\right\}
\le \max \{R(C),R(D)\}.
\]

It follows (again see (\ref{eq7.9.1})) that $R(C\cup D) \le \max \{R(C),R(D)\}$. Combining that with (\ref{eq7.12.2}), one has that (\ref{eq7.12.1}) holds. That completes the proof of Lemma \ref{lm7.12}.
\end{proof}

\begin{lemma}\label{lm7.13}
Suppose $C$ and $D$ are sets $\in \cS$ such that $C\subset D$. If $r(J_D) <r(J_C)$, then $r(J_{D-C}) = r(J_C)$.
\end{lemma}

\begin{proof}
If instead $r(J_{D-C}) \not= r(J_C)$, then by the hypothesis here, together with Lemma \ref{lm7.11}(III),
\[
r(J_D) = \max\{r(J_{D-C}), r(J_C)\} \ge r(J_C) >r(J_D),
\]
a contradiction.

Hence $r(J_{D-C}) = r(J_C)$ after all. Thus Lemma \ref{lm7.13} holds.
\end{proof}

\textbf{Note.} In Lemma \ref{lm7.16} below, it will be shown that $R(S)<1$. The proof will involve the contrary supposition $R(S) =1$ (see (\ref{eq7.10.1})) and a resulting contradiction. In order to set up that contradiction, Lemma \ref{lm7.14} and Lemma \ref{lm7.15} below will first explore some (ultimately ``self-contradicting'') consequences of (directly or indirectly) the supposition $R(S) =1$.

\begin{lemma}\label{lm7.14}
In Context \ref{ctxt7.1} (that is, Context \ref{ctxt5.1} together with eq.\ (\ref{eq7.8.2})), the following holds:

Suppose $0<\varepsilon <1$. Suppose $D \in \cS$ and $R(D) =1$. Then there exist sets $A$, $B\in \cS$ such that
\begin{equation}\label{eq7.14.1}
\mbox{$A$ and $B$ are disjoint, $A\cup B \subset D$,}\,\,r(J_A)> 1-\varepsilon, \,\,\,\mbox{and $R(B) =1$.}
\end{equation}
\end{lemma}

\begin{proof}
Applying (\ref{eq7.9.1}) and the hypothesis $R(D) =1$ here, let $C\in \cS$ be such that
\begin{equation}\label{eq7.14.2}
C\subset D \,\,\,\mbox{and}\,\,\,r(J_C) >1-\varepsilon.
\end{equation}
Refer to (\ref{eq7.10.1}).  The rest of the argument for Lemma \ref{lm7.14} will be divided into two cases according to whether $R(D-C) =1$ or $R(D-C) <1$.

\textbf{Case 1:} $R(D-C) =1$.  Let $A=C$ and $B=D-C$. Then (\ref{eq7.14.1}) holds by (\ref{eq7.14.2}). That completes the argument for Case 1.

\textbf{Case 2:} $R(D-C) <1$.

If $R(C) <1$ were to  hold, then by (\ref{eq7.14.2}) and Lemma \ref{lm7.12}, one would have $R(D) = \max \{R(C),R(D-C)\}<1$, contradicting the hypothesis that $R(D) =1$.  Hence instead one has that
\begin{equation}
\label{eq7.14.3}
R(C) =1.
\end{equation}

Recall from (\ref{eq7.8.2}) that $r(J_C) <1$. Applying (\ref{eq7.14.3}) and (\ref{eq7.9.1}), let $G \in \cS$ be such that
\begin{equation}\label{eq7.14.4}
G\subset C\,\,\,\mbox{and}\,\,\,r(J_G) >r(J_C).
\end{equation}

Refer to the first part of (\ref{eq7.14.4}). Our next task is to show that
\begin{equation}\label{eq7.14.5}
r(J_{C-G}) = r(J_G).
\end{equation}
If instead $r(J_{C-G}) \not= r(J_G)$ were to hold, then by Lemma \ref{lm7.11}(III) and then (\ref{eq7.14.4}), one would have
\[
r(J_C) = \max\left\{ r(J_{C-G}), r(J_G) \right\}\ge r(J_G) >r(J_C),\nonumber
\]
a contradiction.  Hence (\ref{eq7.14.5}) holds.

Combining (\ref{eq7.14.5}) with (\ref{eq7.14.4}) and (\ref{eq7.14.2}), one now has that
\begin{equation}\label{eq7.14.6}
G\subset C \subset D \,\,\mbox{and}\,\,r(J_{C-G}) = r(J_G)> r(J_C) >1-\varepsilon.
\end{equation}

Referring to (\ref{eq7.14.3}) and (\ref{eq7.14.4}) (its ``inclusion'' $G\subset C$, also listed in (\ref{eq7.14.6})), one has by Lemma \ref{lm7.12} that
\[
1= R(C) = \max\{R(C-G),R(G)\}
\]
and hence either $R(C-G) =1$ or $R(G) =1$.

If $R(G) =1$, then let $A= C-G$ and $B=G$. If instead $R(G)<1$ and (hence) $R(C-G) =1$, then instead let $A=G$ and $B=C-G$. In either case, $A\cup B = C\subset D$ by (\ref{eq7.14.6}), and in fact by (\ref{eq7.14.6}) all parts of (\ref{eq7.14.1}) are satisfied. That completes the argument for Case 2.

The proof of Lemma \ref{lm7.14} is complete.
\end{proof}

\begin{lemma}\label{lm7.15} In Context \ref{ctxt7.1} (that is, Context \ref{ctxt5.1} together with eq.\ (\ref{eq7.8.2})), the following holds:

If $R(S) =1$, then there exists a sequence $(A_1,A_2,A_3, \dots)$ of (pairwise) disjoint sets, with $A_k \in \cS$ for each $k\in \N$, such that -- writing $A_k$ also as $A(k)$ for typographical convenience -- one has that $\lim_{k\to\infty} r(J_{A(k)}) =1$.
\end{lemma}

\begin{proof}
As in the statement of Lemma \ref{lm7.15}, the notations $A_k$ and $A(k)$ will mean the same thing.

We shall recursively define two sequences $(A_1,A_2,A_3,\dots)$ and $(B_0,B_1,B_2,B_3,\dots)$ of sets $\in \cS$ such that for each $k\in \N$,
\begin{align}
&\mbox{the sets $A_k$ and $B_k$ are disjoint,}\label{eq7.15.1}\\
&A_k\cup B_k \subset B_{k-1} \label{eq7.15.2}\\
&r\left( J_{A(k)}\right) > 1-2^{-k},\,\,\mbox{and}\label{eq7.15.3}\\
&R(B_k) =1. \label{eq7.15.4}
\end{align}

To start off, define the set $B_0 \in \cS$ by $B_0 =S$. Then by hypothesis, $R(B_0) = R(S) =1$.

Now for the recursion step, suppose $k$ is a positive integer, and the set $B_{k-1} \in \cS$ is already defined, such that $R(B_{k-1})=1$. Applying Lemma \ref{lm7.14}, let $A_k$, $B_k\in \cS$ be such that (\ref{eq7.15.1}), (\ref{eq7.15.2}), (\ref{eq7.15.3}), and (\ref{eq7.15.4}) hold.

That completes the recursive definition of the sets $A_k$, $k\ge 1$ and $B_k$, $k\ge 0$.

By (\ref{eq7.15.3}), $\lim_{k\to\infty} r(J_{A(k)}) =1$. To complete the proof of Lemma \ref{lm7.15}, all that remains is to show that the sets $A_k$, $k\in \N$ are (pairwise) disjoint.

Suppose $j$ and $\ell$ are positive integers such that $j<\ell$. It will suffice to show that $A_j$ and $A_\ell$ are disjoint.

By (\ref{eq7.15.2}), one has that $B_0 \supset B_1 \supset B_2 \supset B_3 \supset \dots\,.$ Hence $B_j\supset B_{\ell-1}$ (with equality if $\ell = j+1$).  Hence by (\ref{eq7.15.2}), $A_\ell \subset B_{\ell-1} \subset B_j$. By (\ref{eq7.15.1}), the sets $A_j$ and $B_j$ are disjoint. Hence the sets $A_j$ and $A_\ell$ are disjoint. That completes the proof of Lemma \ref{lm7.15}.
\end{proof}

\begin{lemma}\label{lm7.16}
In Context \ref{ctxt7.1} (that is, Context \ref{ctxt5.1} together with eq.\ (\ref{eq7.8.2})), the following holds (see (\ref{eq7.9.1})):
\begin{equation}\label{eq7.16.1}
R(S) <1.
\end{equation}
\end{lemma}

\textbf{Proof.} Suppose instead (see (\ref{eq7.10.1})) that
\begin{equation}\label{eq7.16.2}
R(S) =1.
\end{equation}
We shall aim for a contradiction.

The argument will be divided into several ``steps'', followed by a ``Claim 0'' and then a brief ``Conclusion of proof''.

\medskip
\textbf{Step 1.} Applying (\ref{eq7.16.2}) and Lemma \ref{lm7.15}, let $(C_1,C_2,C_3,\dots)$ be a sequence of sets $\in \cS$ such that
\begin{equation}\label{eq7.16.3}
\mbox{The sets $C_k$, $k\in \N$ are (pairwise) disjoint, and}
\end{equation}
\begin{equation}\label{eq7.16.4}
\lim_{k\to\infty} r(J_{C(k)}) =1.
\end{equation}
Here and below, the notation $C(k)$ means $C_k$, and is used in subscripts for typographical convenience.

By (\ref{eq7.16.3}),
\begin{equation}\label{eq7.16.5}
\sum^\infty_{k=1} \mu(C_k) = \mu \left( \bigcup^\infty_{k=1}C_k\right) \le \mu(S) =1.
\end{equation}
Thus the sum in (\ref{eq7.16.5}) converges.  Hence
\begin{equation}\label{eq7.16.6}
\lim_{k\to\infty}\mu(C_k) =0.
\end{equation}

\textbf{Step 2.} For each positive integer $k$, we shall define the following items:

\noindent (i) positive numbers, $t_k$ and $u_k$,

\noindent (ii) a positive integer $\zeta_k$ (also denoted $\zeta(k))$,

\noindent (iii) a set $A_k$ (also denoted $A(k))$ $\in \cS$, and

\noindent (iv) a positive integer $n_k$ (also denoted $n(k))$,

\noindent such that the following hold:
\begin{equation}\label{eq7.16.7}
\mbox{For all $k\ge 1$,}\quad 1 - \left(1/2^k\right) <t_k<u_k<1;
\end{equation}
\begin{equation}\label{eq7.16.8}
\mbox{for all $k\ge 1$,}\quad A_k = C_{\zeta(k)};
\end{equation}
\begin{equation}\label{eq7.16.9}
\mu(A_1) \le 1/4;
\end{equation}
\begin{equation}\label{eq7.16.10}
\mbox{for all $k\ge 2$,}\quad \mu(A_k) \le \min_{1\le i\le k-1} \left[\frac{1}{4^k} \cdot \frac{1}{20} \cdot u_i^{\doub(n(i))}\right];
\end{equation}
\begin{equation}\label{eq7.16.11}
\mbox{for all $k\ge 1$,}\quad r\left(J_{A(k)}\right) >u_k;
\end{equation}
\begin{equation}\label{eq7.16.12}
\mbox{for all $k\ge 1$,}\quad (u_k/t_k)^{\doub(n(k))} > 20\cdot k;\,\,\mbox{and}
\end{equation}
\begin{equation}\label{eq7.16.13}
\mbox{for all $k\ge 1$,}\quad \left( E\left[ J_{A(k)}(X_0)\cdot J_{A(k)}\left(X_{\doub(n(k))}\right)\right]\right)^{1/\doub(n(k))} > u_k.
\end{equation}
The definition is recursive and is as follows:

We start with $k=1$. Define the positive numbers $t_1$ and $u_1$ by $t_1 = 3/5$ and $u_1 =2/3$. Then (\ref{eq7.16.7}) holds for $k=1$. (That is just, with convenient slight abuse of grammar, a short way of saying that for $k=1$, the three inequalities in (\ref{eq7.16.7}) hold. Similar slight abuses of grammar are employed below.) Applying (\ref{eq7.16.4}) and (\ref{eq7.16.6}) and the fact that $u_1 = 2/3 <1$, let $\zeta_1$ be a positive integer such that $\mu(C_{\zeta(1)})\le 1/4$ and $r(J_{C(\zeta(1))}) > u_1$; and then define the set $A_1 \in \cS$ by $A_1:= C_{\zeta(1)}$.

Then (\ref{eq7.16.8}) and (\ref{eq7.16.11}) hold for $k=1$, and (\ref{eq7.16.9}) holds. Keeping in mind that $u_1/t_1 = (2/3)/(3/5) = 10/9>1$, and referring to (\ref{eq7.3.1}) and to (\ref{eq7.16.11}) for $k=1$, let $n_1$ be a positive integer sufficiently large that (\ref{eq7.16.12}) and (\ref{eq7.16.13}) hold for $k=1$. Now (\ref{eq7.16.9}) holds; and all of (\ref{eq7.16.7}), (\ref{eq7.16.8}), (\ref{eq7.16.11}), (\ref{eq7.16.12}), and (\ref{eq7.16.13}) holds for $k=1$. The definition for $k=1$ is complete.

Now here is the recursion step. Suppose $\ell\in \N$ is such that $\ell \ge 2$. Suppose that for each $k\in \{1,2,\dots,\ell-1\}$, the positive numbers $t_k$ and $u_k$, the positive integer $\zeta_k$, the set $A_k \in \cS$, and the positive integer $n_k$ have already been defined, such that (\ref{eq7.16.7}), (\ref{eq7.16.8}), (\ref{eq7.16.11}), (\ref{eq7.16.12}), and (\ref{eq7.16.13}) hold for all $k\in \{1,2,\dots,\ell-1\}$, (\ref{eq7.16.9}) holds, and (if $\ell \ge 3$) (\ref{eq7.16.10}) holds for all $k \in \N$ such that $2 \le k\le \ell-1$.

Referring   to (\ref{eq7.8.2}), let $t_\ell$ be a positive number such that
\begin{equation}\label{eq7.16.14}
\max\left\{1- (1/2)^\ell, r(J_{A(\ell-1)})\right\} < t_\ell <1,
\end{equation}
and then let $u_\ell$ be a positive number such that $t_\ell <u_\ell<1$. Then (\ref{eq7.16.7}) holds for $k=\ell$.

Recall that $0<u_k <1$ for all $k\in \{1,2,\dots, \ell-1\}$ by the ``recursion assumption'' of (\ref{eq7.16.7}) for $k\in \{1,2,\dots,\ell-1\}$, and that $0< u_\ell <1$ as well from (\ref{eq7.16.7}) for $k=\ell$ established just above.

Applying those observations and (\ref{eq7.16.4}) and (\ref{eq7.16.6}), let $\zeta_\ell$ be a positive integer sufficiently large that
\begin{equation}\label{eq7.16.15}
\zeta_\ell >\zeta_{\ell-1},
\end{equation}
\begin{equation}\label{eq7.16.16}
r\left( J_{C(\zeta(\ell))}\right) > u_\ell,\quad\mbox{and}
\end{equation}
\begin{equation}\label{eq7.16.17}
\mu\left( C_{\zeta(\ell)}\right) \le \min_{1\le i\le \ell-1} \left( \frac{1}{4^\ell} \cdot \frac{1}{20} \cdot u^{\doub(n(i))}_i\right).
\end{equation}
Then define the set $A_\ell \in \cS$ by $A_\ell := C_{\zeta(\ell)}$.  By (\ref{eq7.16.16}) and (\ref{eq7.16.17}), eqs.\ (\ref{eq7.16.8}), (\ref{eq7.16.10}), and (\ref{eq7.16.11}) hold for $k=\ell$.

Finally, keeping in mind that $u_\ell/t_\ell >1$ by (\ref{eq7.16.7}) for $k=\ell$ (again, already established a little bit above), and referring to (\ref{eq7.3.1}) and to (\ref{eq7.16.11}) for $k=\ell$ (established just above), let $n_\ell$ be a positive integer sufficiently large that (\ref{eq7.16.12}) and (\ref{eq7.16.13}) hold for $k= \ell$. Then all of (\ref{eq7.16.7}), (\ref{eq7.16.8}), (\ref{eq7.16.10}),  (\ref{eq7.16.11}), (\ref{eq7.16.12}), and (\ref{eq7.16.13}) holds for $k=\ell$.

That completes the recursion step.

The items $t_k$, $u_k$, $\zeta_k$, $A_k$, and $n_k$ have been recursively defined for all $k\in \N$.
\medskip

\textbf{Step 3.} Some simple technical observations will be pertinent to what follows.

By (\ref{eq7.16.15}), $\zeta_1<\zeta_2<\zeta_3<\dots$, and hence those positive integers are distinct. Hence by (\ref{eq7.16.3}), the sets $C_{\eta(1)}$, $C_{\eta(2)}$, $C_{\eta(3)},\dots$ are (pairwise) disjoint. Thus by (\ref{eq7.16.8}),
\begin{equation}\label{eq7.16.18}
\mbox{The sets $A_k$, $k\in\N$ are (pairwise) disjoint.}
\end{equation}

Also, by (\ref{eq7.16.9}) and (\ref{eq7.16.10}) (and the inequalities $0<u_k<1$ for $k\in \N$ from (\ref{eq7.16.7})) one has that $\mu(A_k) \le 1/4^k$ for each $k\in \N$. Consequently, of course,
\begin{equation}\label{eq7.16.19}
\sum^\infty_{k=1} \mu(A_k) \le \sum^\infty_{k=1} 1/4^k = 1/3 <1.
\end{equation}
Also, by (\ref{eq7.7.5}) (with $\ell =0$ and $B=A$ there) and again the sentence just before (\ref{eq7.16.19}), one has that for each $k\in\N$,
\[
E\left[\left(J_{A(k)}(X_0)\right)^2\right] \le \mu (A_k) \le 1/4^k.
\]
Hence
\begin{equation}\label{eq7.16.20}
\sum^\infty_{k=1} \left\| J_{A(k)}(X_0) \right\|_2 \le \sum^\infty_{k=1} 1/2^k =1<\infty.
\end{equation}

For any nonempty set $\Gamma \subset \N$, for the set $G:= \bigcup_{k\in \Gamma}A_k$,  by (\ref{eq7.16.18}), one has that $\mu(G) = \sum_{k\in \Gamma} \mu(A_k)$, and that for any $s\in S$, $I_G (s) = \sum_{k\in\Gamma}I_{A(k)}(s)$, and hence for any $s\in S$,  by (\ref{eq7.6.1}),
\begin{equation}\label{eq7.16.21}
J_G(s) = I_G(s) - \mu(G) =\sum_{k\in \Gamma} (I_{A(k)}(s) - \mu(A_k)) =\sum_{k\in \Gamma}J_{A(k)}(s).
\end{equation}

One other technical observation will be useful later on. By (\ref{eq7.16.7}), (\ref{eq7.16.11}), and (\ref{eq7.16.14}) (for $\ell \ge 2$), one has that
\[
t_1<u_1<r(J_{A(1)}) <t_2<u_2<r(J_{A(2)}) <t_3<u_3<r(J_{A(3)})< \dots\, .
\]
In particular, for any given $k\ge 2$,
\begin{equation}\label{eq7.16.22}
r(J_{A(1)}) <r(J_{A(2)})<\dots <r(J_{A(k-1)}) < t_k < u_k < r(J_{A(k)}) <r(J_{A(k+1)}) <r(J_{A(k+2)})<\dots\, .
\end{equation}

\textbf{Step 4.} Now define the set $B\in \cS$ by
\begin{equation}\label{eq7.16.23}
B = \bigcup^\infty_{k=1} A_k.
\end{equation}
By (\ref{eq7.16.21}) with $\Gamma = \N$ and (hence) $G=B$, one has that for any $s\in S$,
\begin{equation}\label{eq7.16.24}
J_B(s) = \sum^\infty_{k=1} J_{A(k)}(s).
\end{equation}

Our goal now is to show (under the assumption of (\ref{eq7.16.2})) that $r(J_B) =1$, thereby bringing about a contradiction to (\ref{eq7.8.2}). The main work for that purpose will be done in Claim 0 below.

\medskip
\textbf{Claim 0.} Suppose $k$ is an integer such that $k\ge 2$. Then
\begin{equation}\label{eq7.16.25}
\left( E\left[J_B(X_0) \cdot J_B(X_{\doub(n(k))})\right]\right)^{1/\doub(n(k))} >t_k.
\end{equation}

\textbf{Proof of Claim 0.}
The argument for Claim 0 will be divided into ``steps'' labeled Step A, Step B, $\dots$ , Step F.

\textbf{Step A.} For the given fixed integer $k\ge 2$ in the statement of Claim 0, define the sets $B(1)$ and $B(2)\in \cS$ as follows:
\begin{equation}\label{eq7.16.26}
B(1) = \bigcup^{k-1}_{i=1} A_i \quad\mbox{and}\quad B(2) = \bigcup^\infty_{i=k+1}A_i.
\end{equation}

By (\ref{eq7.16.21}) with $\Gamma = \{1,2,\dots,k-1\}$ (and hence $G=B(1)$ from (\ref{eq7.16.26})), one has that for any $s\in S$,
\begin{equation}\label{eq7.16.27}
J_{B(1)}(s) = \sum^{k-1}_{i=1}J_{A(i)}(s).
\end{equation}
Similarly, by (\ref{eq7.16.21}) with $\Gamma = \{k+1,k+2,k+3,\dots\,\}$ (and hence $G = B(2)$ from (\ref{eq7.16.26})), one has that for any $s\in S$,
\begin{equation}\label{eq7.16.28}
J_{B(2)}(s) = \sum^\infty_{i=k+1} J_{A(i)}(s).
\end{equation}
By (\ref{eq7.16.24}), (\ref{eq7.16.27}), and (\ref{eq7.16.28}), one has that for any $s\in S$,
\begin{equation}\label{eq7.16.29}
J_B(s) = J_{B(1)}(s) + J_{A(k)}(s) + J_{B(2)}(s).
\end{equation}
(Alternatively by (\ref{eq7.16.18}) and (\ref{eq7.16.26}), one has that the sets $B(1)$, $A(k)$, and $B(2)$ are (pairwise) disjoint, and their union is the set  $B$; and from that, eq.\ (\ref{eq7.16.29}) holds by a more direct argument analogous to the argument for eq.\ (\ref{eq7.16.21}) itself.)

\textbf{Step B.} By (\ref{eq7.16.29}), one has that
\begin{align}\label{eq7.16.30}
&E\left[J_B(X_0) \cdot J_B\left( X_{\doub(n(k))}\right)\right]\nonumber\\
&=E\left[\left(J_{B(1)}(X_0) + J_{A(k)}(X_0) + J_{B(2)}(X_0)\right) \cdot \left(J_{B(1)}\bigl(X_{\doub(n(k))}\bigr)+ J_{A(k)}\bigl( X_{\doub(n(k))}\bigr) +J_{B(2)}\bigl(X_{\doub(n(k))}\bigr)\right)\right]\nonumber\\
&=E\left[ J_{B(1)}(X_0) \cdot J_{B(1)}\left(X_{\doub(n(k))}\right)\right] + E\left[ J_{B(1)}(X_0) \cdot J_{A(k)}\left(X_{\doub(n(k))}\right)\right] \nonumber\\
&\quad+E\left[J_{B(1)}(X_0) \cdot J_{B(2)}\left(X_{\doub(n(k))}\right)\right] + E\left[ J_{A(k)}(X_0) \cdot J_{B(1)} \left(X_{\doub(n(k))}\right)\right]  \nonumber\\
&\quad + E\left[J_{A(k)}(X_0) \cdot J_{A(k)}\left(X_{\doub(n(k))}\right)\right] + E\left[J_{A(k)}(X_0) \cdot J_{B(2)}\left(X_{\doub(n(k))}\right)\right] \nonumber\\
&\quad+ E\left[ J_{B(2)}(X_0) \cdot J_{B(1)}\left(X_{\doub(n(k))}\right)\right] + E\left[ J_{B(2)}X_0 \cdot J_{A(k)}\left(X_{\doub(n(k))}\right)\right] \nonumber\\
&\quad+ E\left[ J_{B(2)}(X_0) \cdot J_{B(2)}\left(X_{\doub(n(k))}\right)\right].
\end{align}

\textbf{Step C.} For any set $D\in \cS$, by (\ref{eq7.16.27}), then (\ref{eq7.7.1}) and (\ref{eq5.3.2}), then ((\ref{eq7.7.1}) again and)  (\ref{eq7.3.1}), and then (\ref{eq7.16.22}) and finally (\ref{eq7.16.12}),
\begin{align}\label{eq7.16.31}
&\left|E \left[ J_{B(1)}(X_0) \cdot J_D\left( X_{\doub(n(k))}\right)\right] \right| =\left| E\left[\left( \sum^{k-1}_{\ell =1} J_{A(\ell)}(X_0)\right) \cdot J_D\left(X_{\doub(n(k))}\right)\right]\right| \nonumber\\
&= \left|\sum^{k-1}_{\ell=1} E\left[J_{A(\ell)}(X_0) \cdot J_D\left(X_{\doub(n(k))}\right)\right]\right| \le \sum^{k-1}_{\ell=1} \left| E\left[J_{A(\ell)}(X_0) \cdot J_D\left(X_{\doub(n(k))}\right)\right]\right| \nonumber\\
&\le \sum^{k-1}_{\ell=1} \left(E\left[J_{A(\ell)}(X_0) \cdot J_{A(\ell)}\left(X_{2\cdot \doub(n(k))}\right)\right]\right)^{1/2} \nonumber\\
&\le \sum^{k-1}_{\ell=1}\left(\left[r(J_{A(\ell)})\right]^{2\cdot \doub(n(k))}\right)^{1/2} = \sum^{k-1}_{\ell=1}\left[r(J_{A(\ell)})\right]^{\doub(n(k))} \nonumber\\
&\le \sum^{k-1}_{\ell=1} t^{\doub(n(k))}_k = (k-1) \cdot t_k^{\doub(n(k))} \le \frac{1}{20} \cdot u^{\doub(n(k))}_k.
\end{align}

\textbf{Step D.} For any set $D\in \cS$, by (\ref{eq7.16.20}),
\begin{align*}
&\sum^\infty_{\ell=k+1} E\left|J_{A(\ell)}(X_0) \cdot J_D\left(X_{\doub(n(k))}\right)\right|\le \sum^\infty_{\ell = k+1} \left[\left\| J_{A(\ell)}(X_0) \right\|_2 \cdot \left\|J_D(X_{\doub(n(k))})\right\|_2\right]\\
&= \left\|J_D\left(X_{\doub(n(k))}\right)\right\|_2 \cdot \sum^\infty_{\ell= k+1} \left\|J_{A(\ell)}(X_0)\right\|_2 <\infty.
\end{align*}
Hence by (\ref{eq7.16.28}), (\ref{eq7.7.5}), and (\ref{eq7.16.10}), for any set $D\in \cS$,
\begin{align}\label{eq7.16.32}
&\left| E\left[ J_{B(2)}(X_0) \cdot J_D\left(X_{doub(n(k))}\right)\right]\right| = \left|E\left[\left(\sum^\infty_{\ell=k+1} J_{A(\ell)}(X_0) \right) \cdot \left(J_D\left(X_{\doub(n(k))}\right)\right)\right]\right| \nonumber\\
&=\left|\sum^\infty_{\ell = k+1} E\left[ J_{A(\ell)} (X_0) \cdot J_D\left(X_{\doub(n(k))}\right)\right]\right|
\le \sum^\infty_{\ell=k+1} \left|  E\left[J_{A(\ell)}(X_0) \cdot J_D\left(X_{\doub(n(k))}\right)\right]\right| \nonumber\\
&\le \sum^\infty_{\ell=k+1} \mu(A_\ell) \le \sum^\infty_{\ell=k+1} \left[\frac{1}{4^\ell} \cdot \frac{1}{20} \cdot u^{\doub(n(k))}_k\right] \le \frac{1}{20} \cdot u^{\doub(n(k))}_k.
\end{align}

\textbf{Step E.} By strict stationarity and reversibility, the random vectors $(X_0, X_{\doub(n(k))})$ and $(X_{\doub(n(k))}, X_0)$ have the same distribution (on $(S^2, \cS^2))$. Hence for any set $D\in \cS$, one has by (\ref{eq7.16.31}) that
\begin{equation}\label{eq7.16.33}
\left| E\left[ J_D(X_0) \cdot J_{B(1)}(X_{\doub(n(k))})\right]\right| =\left|E\left[J_D(X_{\doub(n(k))})\cdot J_{B(1)}(X_0)\right]\right| \le \frac{1}{20} u^{\doub(n(k))}_k.
\end{equation}
Similarly, for any set $D\in \cS$, one has by (\ref{eq7.16.32}) that
\begin{equation}\label{eq7.16.34}
\left|E\left[J_D(X_0) \cdot J_{B(2)}\left(X_{\doub(n(k))}\right)\right]\right| =\left|E\left[J_D\left(X_{\doub(n(k))}\right) \cdot J_{B(2)}(X_0)\right]\right|\le \frac{1}{20} u^{\doub(n(k))}_k.
\end{equation}
By (\ref{eq7.16.31}), (\ref{eq7.16.32}), (\ref{eq7.16.33}), and (\ref{eq7.16.34}), of the nine terms in the far right side of (\ref{eq7.16.30}), eight of them are each bounded in absolute value by $(1/20)\cdot u^{\doub(n(k))}_k$. The lone exception there is the fifth term, $E[J_{A(k)}(X_0) \cdot J_{A(k)}(X_{\doub(n(k))})]$.

Hence by (\ref{eq7.16.30}) itself,
\begin{equation}\label{eq7.16.35}
E\left[J_B(X_0) \cdot J_B\left( X_{\doub(n(k))}\right)\right] \ge E\left[ J_{A(k)}(X_0) \cdot J_{A(k)}\left(X_{\doub(n(k))}\right)\right] - \frac{8}{20} \cdot u^{\doub(n(k))}_k.
\end{equation}

\textbf{Step F.} Now by (\ref{eq7.16.35}) and (\ref{eq7.16.13}),
\begin{equation}\label{eq7.16.36}
E\left[ J_B(X_0) \cdot J_B\left(X_{\doub(n(k))}\right)\right] \ge u^{\doub(n(k))}_k - \frac{8}{20} \cdot u^{\doub(n(k))}_k
> \frac{1}{2} \cdot u^{\doub(n(k))}_k.
\end{equation}
Also, by (\ref{eq7.16.12}), $u^{\doub(n(k))}_k/t^{\doub(n(k))}>2$ and hence $(1/2) u^{\doub(n(k))}_k > t^{\doub(n(k))}_k$. Hence by (\ref{eq7.16.36}),
\[
E\left[J_B(X_0) \cdot J_B\left( X_{\doub(n(k))}\right)\right] > t^{\doub(n(k))}_k.
\]
Hence (\ref{eq7.16.25}) holds. That completes the proof of Claim 0.

\medskip
\textbf{Conclusion of proof of Lemma \ref{lm7.16}.} For each $k\ge 2$, by Claim 0 and (\ref{eq7.3.1}), $r(J_B)> t_k$. By (\ref{eq7.16.7}), $t_k \to 1$ as $k\to\infty$. Hence (see (\ref{eq7.4.1})), $r(J_B) =1$. But that contradicts (\ref{eq7.8.2}). Hence (\ref{eq7.16.2}) must be false. Hence (see (\ref{eq7.10.1})), eq.\ (\ref{eq7.16.1}) must hold after all. That completes the proof of Lemma \ref{lm7.16}.
\hfill$\square$\break

\begin{proposition}\label{pr7.17}
(I) Proposition \ref{pr3.5}(II) holds.

(II) More specifically, if condition (A3) in Theorem \ref{th3.3} holds, then (see (\ref{eq7.16.1})) for any pair of sets $A$, $B \in \cS$ and any positive integer $n$,
\begin{equation}\label{eq7.17.1}
\left| P(\{X_0 \in A\} \cap \{X_{\doub(n)} \in B\}) - \mu(A)\mu(B)\right| \le [R(S)]^{\doub(n)}.
\end{equation}
\end{proposition}

\begin{proof} Obviously by Lemma \ref{lm7.16}, (I) follows from (II) (which gives a stronger conclusion under the same hypothesis). Our task is to prove (II).

\textbf{Proof of (II).} Suppose $A$, $B\in \cS$ and $n\in \N$. They by (\ref{eq7.7.4}), then (\ref{eq7.7.1}) and (\ref{eq5.3.2}), and then (\ref{eq7.3.1}) and finally (\ref{eq7.9.1}),
\begin{align*}
&\left|P\left(\{X_0\in A\}\cap \{X_{\doub(n)}\in B\}\right) - \mu(A)\mu(B)\right|
=\left| E\left[J_A(X_0) \cdot J_B(X_{\doub(n)})\right]\right|\\
&\le \left(E\left[ J_A(X_0) \cdot J_A(X_{\doub(n+1)})\right]\right)^{1/2} \le \left([r(J_A)]^{\doub(n+1)}\right)^{1/2}
=[r(J_A)]^{\doub(n)}\le [R(S)]^{\doub(n)}.
\end{align*}
That is, (\ref{eq7.17.1}) holds. That completes the proof of Statement (II) in Proposition \ref{pr7.17}.

That completes the proof of Proposition \ref{pr7.17}; and in particular that completes (our review of) the proof of Proposition \ref{pr3.5}(II) (which is Statement (I) in Proposition \ref{pr7.17}).
\end{proof}

\end{document}